\documentclass{amsart}
\usepackage{amssymb}
\usepackage{longtable}
\newtheorem{theorem}{Theorem}[section]

\newtheorem{corollary}[theorem]{Corollary}
\newtheorem{proposition}[theorem]{Proposition}

\theoremstyle{definition}
\newtheorem{definition}[theorem]{Definition}

\theoremstyle{remark}
\newtheorem{remark}[theorem]{Remark}

\numberwithin{equation}{section}

\begin{document}

\title[Primitive complete normal bases] 
{Primitive complete normal bases for regular extensions: Exceptional cyclotomic modules} 

\author{Dirk Hachenberger} 
\address{Institut f\"ur Mathematik, Universit\"at Augsburg, D-86135 Augsburg}
\email{hachenberger@math.uni-augsburg.de}


\date{10 December 2019}

\newcommand{\cfq}{\overline{\mathbb{F}}_q}
\newcommand{\F}{\mathbb{F}} 
\newcommand{\K}{\mathbb{K}} 
\newcommand{\fq}{\mathbb{F}_q} 
\newcommand{\fp}{\mathbb{F}_p} 
\newcommand{\fqm}{\mathbb{F}_{q^m}} 
\newcommand{\fqd}{\mathbb{F}_{q^d}} 
\newcommand{\fqn}{\mathbb{F}_{q^n}} 
\newcommand{\fqk}{\mathbb{F}_{q^k}} 
\newcommand{\fql}{\mathbb{F}_{q^\ell}}
\newcommand{\N}{\mathbb{N}}
\newcommand{\Z}{\mathbb{Z}} 
\newcommand{\C}{\mathbb{C}} 
\newcommand{\ft}[1]{\mathbb{F}_{q^{#1}}} 
\newcommand{\tr}{{\rm Tr}} 

\newcommand{\eps}{\varepsilon}
\newcommand{\opp}{{\rm \footnotesize opp}}

\newcommand{\Ord}{{\rm Ord}} 
\newcommand{\Tr}{{\rm Tr}} 
\newcommand{\No}{{\rm No}} 

\newcommand{\ord}{{\rm ord}}

\newcommand{\grp}{{\rm grp}}
\newcommand{\id}{{\rm id}}

\newcommand{\rad}{{\rm rad}}

\newcommand{\sord}{{\rm subord}}
\renewcommand{\char}{{\rm char}}
\renewcommand{\span}{{\rm span}}
\renewcommand{\max}{{\rm max}}
\renewcommand{\min}{{\rm min}}
\newcommand{\equi}{\Leftrightarrow}
\newcommand{\impli}{\Rightarrow}
\newcommand{\lcm}{{\rm lcm}}
\renewcommand{\ker}{{\rm kern}}
\newcommand{\im}{{\rm im}}
\newcommand{\supp}{{\rm Tr}}
\newcommand{\sgn}{\mathop{\rm sgn}}
\renewcommand{\log}{{\rm log}}
\renewcommand{\det}{{\rm det}} 
\renewcommand{\deg}{{\rm deg}}
\renewcommand{\dim}{\mathop{\rm dim}}
\newcommand{\DIV}{\, {\tt div}\, }
\newcommand{\MOD}{\, {\tt mod}\, }

\begin{abstract} A primitive completely normal element for an extension 
$\mathbb{F}_{q^n}/\mathbb{F}_{q}$ of Galois fields is a generator of the multiplicative group of $\mathbb{F}_{q^n}$, which simultaneously 
is normal over every intermediate field of that extension. 
We are going to prove that such a generator exists when $\mathbb{F}_{q^n}/\mathbb{F}_{q}$ 
is an \textit{exceptional} regular extension. 
In combination with \cite{Ha-2001} our investigations altogether settle the existence of primitive completely normal bases 
for any regular extension. 
An important feature of the class of regular extensions is that they comprise every extension of prime power degree. 
\end{abstract}

\maketitle

\section{Introduction}  
\label{section-01} 
For a pair $(q,n)$, where $q>1$ is a prime power and $n\geq 1$ an integer we consider the corresponding extension $\fqn/\fq$ of 
Galois fields. Let $\sigma:v \mapsto v^q$ be the Frobenius automorphism, defined on $\fqn$. 
Then $\sigma$ generates the (cyclic) Galois group of $\fqn/\fq$. 
The intermediate fields of $\fqn/\fq$ correspond  
to the divisors of $n$. 
If $d$ is such a divisor, and if $v\in \fqn$, then $v$ is \textit{normal} over $\fqd$, if its conjugates 
under the Galois group of $\fqn$ over $\fqd$ 
(that is, $\sigma^{dj}(v)$ for $j=0,...,\frac{n}{d}-1$) form an $\fqd$-basis of $\fqn$. 
If $v$ is even normal over $\fqd$ for every $d\mid n$, then $v$ is called a \textit{completely normal} element of $\fqn/\fq$. 

The \textit{Complete Normal Basis Theorem} (1986, by Blessenohl and Johnsen \cite{Bl-Jo-1986}) states that for every extension 
of Galois fields there exists such a completely normal element. 
In the eighties of the last century, there has been proved another celebrated result: The 
\textit{Primitive Normal Basis Theorem} (1987, Lenstra and Schoof \cite{Le-Sc-1987}) 
says that for every extension 
$\fqn/\fq$ there exists a primitive element (of $\fqn$) that is normal over $\fq$. 
Recall from the basic terminology of finite fields (we refer to Lidl and Niederreiter \cite{Li-Ni-1986}) 
that a \textit{primitive} element of $\fqn$ is a generator of its (cyclic) multiplicative group. 

Seeing these two fundamental theorems side by side, it is only natural to ask whether any extension of Galois fields 
even admits a generator of the multiplicative group which 
simultaneously is normal over every intermediate field. Based on the positive results of a computer 
search\footnote{In \cite{Mo-Mu-1996}, for every pair 
$(q,n)$, with $q \leq 97$ a prime and with $q^n < 10^{50}$, there is tabulated a monic irreducible polynomial 
of degree $n$ over $\fq$, whose roots are primitive and completely normal elements for the corresponding extension 
$\fqn/\fq$.}, Morgan and Mullen \cite{Mo-Mu-1996} 
formulated the conjecture that this is indeed the case for any pair $(q,n)$. 
For example, the roots of the polynomials 
\begin{equation} \label{3-8-and-3-16} 
x^8+x^7+2x^3+2x^2+2 \in \mathbb{F}_3[x] \ {\mbox{ and }} \ x^{16}+x^{15}+2x^6+2x+2 \in \mathbb{F}_3[x] 
\end{equation}  
are such \textit{primitive completely normal bases} for the pairs $(q,n)=(3,8)$ and $(q,n)=(3,16)$, respectively. 

Because of the complicated nature of completely normal elements (we refer to \cite{Ha-1997} for an extensive study, and to 
\cite{Ha-2013} for a recent survey), 
a proof of this conjecture is an extremely difficult task, which, if ever found, may be discovered only 
step by step according to other explorations in the theory of finite fields. 
It is the aim of the present paper to put another piece into the puzzle of a proof of the conjecture of Morgan and Mullen: 
We are going to show the existence of a primitive completely normal element for pairs $(q,n)$ which belong to the class of \textit{exceptional 
regular extensions} (see Section \ref{section-03}). 

Throughout, let $p$ be the characteristic of the underlying fields. 
We write $n=p^a n'$ with $n'$ being the \textit{$p$-free part}, that is, $n'$ is not divisible by $p$.  
Furthermore, let $\pi(n')$ denote the set of distinct prime divisors of $n'$ and 
$\rad(n'):=\prod_{r\in \pi(n')}r$ the \textit{radical} of $n'$ (which is equal to $1$ if $n'=1$). 
The least integer $s\geq 1$ such that $q^s \equiv 1 \MOD \rad(n')$ is denoted by 
$\ord_{\rad(n')}(q)$; it is the \textit{order of $q$ modulo $\rad(n')$}. 
By \cite[Definition 1.3]{Ha-2001}, an extension $\fqn/\fq$, as well as the pair $(q,n)$, are \textit{regular}, provided that $\ord_{\rad(n')}(q)$ and $n$ are relatively prime. 
The description of completely normal elements for regular extensions requires the distinction into two subclasses:   
\begin{itemize} 
\item The class of non-exceptional ones, 
\item and the class of exceptional extensions. 
\end{itemize} 
The difference is explained in Section \ref{section-03} after introducing the concept of a 
cyclotomic module in Section \ref{section-02}. 
At this place it is important to note that $(q,r^m)$ is always regular and non-exceptional (for arbitrary $q$), when $r$ is any odd prime, or when $r=p$. 
The phenomenon of exceptionality however occurs for certain $2$-power extensions, namely when 
$q\equiv 3 \MOD 4$ and $n=2^c$ with 
$c \geq 3$ and when $\ord_{2^c}(q)=2$. 

The main result in \cite{Ha-2001} is as follows: {\em Assume that $(q,n)$ is regular, and further that 
$q\equiv 1 \MOD 4$ if $q$ is odd and $n$ is even. Then there exists a primitive completely normal element in $\fqn/\fq$.} 
A cornerstone of its proof has been the ability to efficiently describe the characteristic 
function of the set of all primitive completely normal elements in such extensions by using the theory of finite field characters.  
The additional assumption ($q\equiv 1 \MOD 4$ if $q$ is odd and $n$ is even) had been chosen to guarantee that the pair under consideration is a non-exceptional one, because, in a sense which will become clear in Section \ref{section-04}, the exceptional 
cases disturb a very pleasant structure which makes their handling much more difficult. 

In the meantime however, and this is a central part of the present contribution, 
we are able to develop an efficient (though more involved) character based description of 
the set of all primitive completely normal elements in exceptional regular extensions as well (see Sections \ref{section-05} and \ref{section-06}, as well 
as Section \ref{section-13} for a further technical detail). 
We assume that $q\equiv 3 \MOD 4$ and that $n$ is even. 
The use of finite field characters leads 
to the sufficient number theoretical existence criterion in Proposition \ref{prop:suff-crit} of Section 
\ref{section-07}. The analysis of this criterion is carried out in Sections \ref{section-08}-\ref{section-10} for the case where $n\equiv 0 \MOD 8$; 
it is satisfied for all $(q,n)$ different from $(3,8)$ or $(3,16)$. 
In Section \ref{section-11} we consider all degrees $n$ with $n \equiv 2 \MOD 4$ or $n \equiv 4 \MOD 8$. 
The particular instances $(3,8)$ and $(3,16)$ are briefly considered in Section \ref{section-12}. 
Our main result is as follows: 

\begin{theorem} \label{thm:main} 
\label{thm:pcn-neu} 
Consider a regular extension $\fqn/\fq$, where $q\equiv 3 \MOD 4$ and where $n$ is even. 
Then there exists a primitive element of $\fqn$ that is completely normal over $\fq$. 
\end{theorem} 

\noindent 
Altogether, this proves the Morgan-Mullen-Conjecture for the whole class of 
regular extensions, whose importance, as mentioned in the abstract, relies on the fact that $(q,n)$ is regular for every $q$ whenever $n$ is any prime power. We therefore have:  
\begin{corollary} 
\label{cor:ppext} 
Let $r$ be any prime, $m \geq 0$ any integer and let $\fq$ be any Galois field. 
Then there exists a primitive element of $\mathbb{F}_{q^{r^m}}$ which is completely normal over $\fq$. \qed 
\end{corollary} 
\noindent 
At this place, we like to mention that the essential breakthrough of Blessenohl and Johnsen's proof \cite{Bl-Jo-1986} was just 
to provide the existence of completely normal elements for pairs $(q,r^m)$ with $r$ a prime. 
After that, if $\prod_{i=1}^s r_i^{m_i}$ is the prime power factorization 
of some $n$, and if $v_i$ is completely normal for the pair $(q,r_i^{m_i})$, a standard argument 
 shows that the product $v:=\prod_{i=1}^s v_i$ gives a completely normal element for $\fqn/\fq$ (see \cite[Hilfssatz 4.4]{Bl-Jo-1986} or \cite[Corollary 4.11]{Ha-1997}). 
However, even if the $v_i$ would additionally be primitive in their extension, 
then $v$ is definitely not primitive in the composed field $\fqn$. This is another reason why 
a \textit{primitive complete normal basis theorem} is much more difficult to prove.

To conclude this introduction, we mention that Blessenohl \cite{Bl-2005}   
has proved the existence of a primitive completely normal element for any pair $(q,2^a)$ with $q\equiv 3 \MOD 4$ and $2^a$ dividing 
$q^2-1$. Another region of $2$-power extensions is considered in \cite{Ha-2010}: If $q\equiv 3 \MOD 4$ and 
if $m\geq e+3$, where $2^e$ is the largest power of $2$ dividing $q^2-1$, then there are at least 
$4 \cdot (q-1)^{2^{m-2}}$ primitive elements in $\mathbb{F}_{q^{2^m}}$ which are completely normal over $\fq$. 
While the proofs in \cite{Bl-2005,Ha-2010} rely on different arguments, the results still leave open 
some $2$-power extensions. However, all pairs $(q,n)$ considered in \cite{Bl-2005,Ha-2010} are 
covered by the present Theorem \ref{thm:main}.

\section{The canonical decomposition of a regular extension}  
\label{section-02} 
For any $d\mid n$ the additive group of $\fqn$ carries the structure of an $\fqd[x]$-module; the operation of 
$f(x) \in \fqd[x]$ on $z\in \fqn$ is given by 
$z \mapsto f(\sigma^d)(z)$. In fact, $\fqn$ is a cyclic $\fqd[x]$-module, and the generators of $\fqn$ in this context are presicely the 
normal elements of $\fqn/\fqd$. 
The \textit{$q^d$-order} of $z\in \fqn$ is the monic polynomial $g(x) \in \fqd[x]$ of least degree such that 
$g(\sigma^d)(z)=0$. It is denoted by $\Ord_{q^d}(z)$, and $z$ is normal over $\fqd$ if and only if $\Ord_{q^d}(z)=x^{n/d}-1$. 

Within the polynomial ring $\fq[x]$ we have the canonical decomposition 
\begin{equation} \label{eqn:2B1} 
x^n-1=(x^{n'}-1)^{p^a}=\prod_{k\mid n'}\Phi_k(x)^{p^a}, 
\end{equation}  
where $\Phi_k(x)\in \fq[x]$ denotes the $k$-th cyclotomic polynomial. 
The coefficients of $\Phi_k(x)$ are elements of the prime field $\mathbb{F}_p$; moreover, 
$\Phi_k(x)^{p^a}=\Phi_k(x^{p^a})$. 
For every $k\mid n'$, we therefore call  
\begin{equation} \label{eqn:cyc-mod} 
{\mathcal C}_k:=\{w \in \fqn: \Phi_k(\sigma)^{p^a}(w)=0 \} 
\end{equation} 
the \textit{cyclotomic module} of $\fqn/\fq$ corresponding to $k$.  
The $\fq$-dimension of ${\mathcal C}_k$ is equal to $p^a \cdot \deg(\Phi_k(x))=p^a \cdot \varphi(k)$, where $\varphi$ is the 
Euler function.
According to (\ref{eqn:2B1}), we obtain the (canonical) decomposition of $\fqn$ into the direct sum of its cyclotomic modules:  
\begin{equation} 
\label{eqn:2B2} 
\fqn=\bigoplus_{k\mid n'} {\mathcal C}_k. \end{equation} 
Consequently, any $z\in \fqn$ can uniquely be written as $\sum_{k\mid n'} z_k$, where $z_k \in {\mathcal C}_k$ for every $k\mid n'$. 
Moreover, $z$ is normal in $\fqn/\fq$ if and only if $z_k$ generates ${\mathcal C}_k$ as $\fq[x]$-module for any $k \mid n'$, and this holds if and only if 
$\Ord_q(z_k)=\Phi_k(x)^{p^a}$ for any $k\mid n'$. 

We are now going to explain what can be said about the components of a completely normal element. 
Consider therefore again a divisor $k$ of $n'$.  
With $\rad(k)$ being the radical of $k$ we have 
\begin{equation} \label{eqn:2C1} 
\Phi_k(x)^{p^a}=\Phi_{\rad(k)} (x^{p^a k/\rad(k)}). 
\end{equation} 
The important number $p^ak/\rad(k)$ is called the \textit{module character} of ${\mathcal C}_k$ (compare with \cite[Definition 5.4.30]{Ha-2013}), 
a notion which is motivated by the fact that ${\mathcal C}_k$ (with respect to $\sigma^d$) is an $\fqd[x]$-submodule of $\fqn$ for every $d$ dividing $p^ak/\rad(k)$. 
Moreover, ${\mathcal C}_k$ is, with respect to any such $d$,  a cyclic $\fqd[x]$-module, and  
$w$ generates ${\mathcal C}_k$ as such if and only if  
\begin{equation} \label{eqn:2C2} 
\Ord_{q^d}(w) = \Phi_{\rad(k)} (x^{p^ak/(\rad(k)d)}). 
\end{equation} 
Now, an element $w\in {\mathcal C}_k$ is a \textit{complete generator} for ${\mathcal C}_k$ over $\fq$, provided that 
$w$ (simultaneously) generates ${\mathcal C}_k$ as an $\fqd[x]$-module for every divisor $d$ of $p^ak/\rad(k)$, which means that 
(\ref{eqn:2C2}) holds for every divisor $d$ of the module character.  
So, considering once more the canonical decomposition of $\fqn$ over $\fq$ in (\ref{eqn:2B2}), it becomes transparent 
that for an element $z=\sum_{k\mid n'}z_k$ to be completely normal, for every $k\mid n'$, the component 
$z_k$ necessarily has to be a complete generator for ${\mathcal C}_k$.  
The converse of that statement is not true in general, but it holds if and only if $n'$ and $\ord_{\rad(n')}(q)$ are relatively prime 
(see \cite[Section 19]{Ha-1997} and also \cite[Theorem 5.4.45]{Ha-2013}). 
We may therefore conclude: 

\begin{proposition} \label{prop:cn-by-can-dec} 
Assume that $(q,n)$ is a regular pair, which means that $n$ and $\ord_{\rad(n')}(q)$ are relatively prime. 
Then $z = \sum_{k\mid n'} z_k \in \fqn$ is completely normal over $\fq$ if and only if any component 
$z_k$ of its canonical decomposition is a complete generator for the cyclotomic module 
${\mathcal C}_k$. \qed 
\end{proposition}

\section{Exceptional and non-exceptional cyclotomic modules}   
\label{section-03} 
In the present section we are going to define the notions of exceptionality and non-exceptionality within the class of regular extensions.  

\begin{definition} 
\label{def:exceptional} Let $(q,n)$ be regular and consider a divisor $k$ of $n'$. 
We write $k=2^c \ell$, where $\ell$ is odd. 
Then the cyclotomic module ${\mathcal C}_k$ is called \textit{exceptional} (over $\fq$), provided the following specific number theoretical conditions are satisfied: 
\begin{equation} 
\label{eqn:3A1} 
{\mbox{ $q \equiv 3 \MOD 4$ and $c \geq 3$ and $\ord_{2^c}(q)=2$.}} 
\end{equation} 
In all other cases, ${\mathcal C}_k$ is \textit{non-exceptional}. 
The notions exceptional and non-exceptional are also used for the divisor $k$. \qed 
\end{definition} 

\begin{remark} 
Given that $(q,n)$ is regular, the entire field extension $\fqn/\fq$ as well as the pair $(q,n)$ are called \textit{exceptional}, 
provided that there exists a $k\mid n'$ such that the cyclotomic module ${\mathcal C}_k$ is an exceptional one. 
On the other hand, $\fqn/\fq$ as well as $(q,n)$ are \textit{non-exceptional}, if $\fqn$ is  
composed by non-exceptional cyclotomic modules over $\fq$, only; 
and this holds if and only if one of the following cases occurs:  
\begin{enumerate} 
\item $q$ is even, or 
\item $q\equiv 1 \MOD 4$, or 
\item $q\equiv 3 \MOD 4$ and $n' \not\equiv 0 \MOD 8$. 
\end{enumerate} 

\noindent 
As discussed in Section \ref{section-01}, 
the class of regular extensions with $q$ even, or with $q\equiv 1 \MOD 4$,  or with $q\equiv 3 \MOD 4$ and $n'$ odd is considered 
in \cite{Ha-2001}, and so we are left here with those regular extensions, where $q\equiv 3\MOD 4$ and $n$ is even (see Theorem 
\ref{thm:main}). 
Because of the possible occurence of exceptional cyclotomic modules, we distinguish these remaining pairs into the following 
two subclasses: 
\begin{itemize} 
\item[(a)] Either $q\equiv 3 \MOD 4$ and $n' \equiv 0 \MOD 8$,  
\item[(b)] or $q\equiv 3 \MOD 4$ and $n' \equiv 2 \MOD 4$ or $n'\equiv 4 \MOD 8$. 
\end{itemize} 
\noindent Exceptional cyclotomic modules occur precisely in (a). \qed 
\end{remark} 

\noindent 
For a regular pair $(q,n)$ with $q \equiv 3 \MOD 4$ and $n$ even, 
we are now going to figure out, which of the divisors $k\mid n'$ are exceptional, and which are not. 
Let therefore  
$n'=2^b\overline{n}$ with $\overline{n}$ being odd (hence $b\geq 1$). Furthermore, let $2^e$ be the maximal power of $2$ dividing $q^2-1$ (then $e\geq 3$), 
and for $j=0,\dots,b$ define  
\begin{equation} 
\label{eqn:3C1} 
D_j:=\{2^j\ell: \ell\mid \overline{n}\}. 
\end{equation}  
giving a partition of the set of all divisors of $n'$. 
We next introduce the sets 
\begin{equation} 
\label{eqn:3C2} 
{\mathcal N}' := \left\{\begin{array}{ll} D_0 \cup D_1, & {\mbox{if $b=1$}} \\ 
                                         D_0 \cup D_1 \cup D_2, & {\mbox{if $b\geq 2$}},\end{array}\right. {\mbox{ and }} \ 
                                         {\mathcal N}'':= \bigcup_{j=e+1}^b D_j {\mbox{ when $b>e$}}, 
                                         \end{equation} 
as well as                                          
\begin{equation} 
\label{eqn:3C3} 
{\mathcal E}  := \bigcup_{j=3}^{\min(b,e)} D_j, \ \ {\mbox{ when $b\geq 3$}}. 
\end{equation}  
Finally, write ${\mathcal N}:={\mathcal N}' \cup {\mathcal N}''$.  
Then, altogether, for $k\mid n'$, the cyclotomic module ${\mathcal C}_k$ is 
\begin{itemize} 
\item exceptional when $k\in {\mathcal E}$, 
\item and non-exceptional when $k\in {\mathcal N}$. 
\end{itemize} 
Observe that $\mathcal E$ and ${\mathcal N}''$ are empty when $b\leq 2$.

\section{Complete generators for cyclotomic modules of regular extensions}   
\label{section-04} 
The aim of the present section is to provide a strengthening as well as a refinement of Proposition \ref{prop:cn-by-can-dec}. 
For an integer $k$ which is relatively prime to $q$, the 
\textit{sub-order} of $q$ modulo $k$ is defined to be 
\begin{equation} 
\label{eqn:3B1} 
\sord_k(q)  := \frac{\ord_k(q)}{\ord_{\rad(k)}(q)} . 
\end{equation} 
Observe that $\sord_k(q)=\ord_k(q^u)$, when $u=\ord_{\rad(k)}(q)$;  
moreover, $\sord_k(q)$ is a divisor of $k/\rad(k)$, and therefore only composed 
of primes dividing $k$ (see \cite[Section 19]{Ha-1997}). We therefore write 
\begin{equation} 
\label{eqn:3B2} 
\sord_k(q)=\prod_{r\in \pi(k)} r^{\alpha(r)}, 
\end{equation} 
where $\alpha(r)\geq 0$ for all $r\in \pi(k)$, and where $\pi(k)$ as before denotes the set of prime divisors of $k$ 
(with the convention that $\sord_k(q)=1$, if $k=1$, that is, when $\pi(k)$ is empty).

\begin{definition} 
\label{def:central-index} 
Let $(q,n)$ be a regular pair and for any $k\mid n'$, using (\ref{eqn:3B2}), let 
\begin{equation} 
\label{eqn:3B3} 
\tau_k:=\prod_{r\in \pi(k)} r^{\lfloor \alpha(r)/2 \rfloor}, 
\end{equation} 
(where $\lfloor \rho \rfloor$ denotes the integral part of $\rho$). 
Then $\tau_k$ is called the {\em central index} of the corresponding cyclotomic module ${\mathcal C}_k$. \qed 
\end{definition} 

\noindent 
By definition, the central index $\tau_k$ divides $k/\rad(k)$. 
If $k$ is exceptional, then $\tau_k$ is odd and therefore even 
$2\tau_k$ divides $k/\rad(k)$. The important role of the central index $\tau_k$ lies in the following result, which is the announced 
strengthening of Proposition \ref{prop:cn-by-can-dec} (see \cite[Section 20]{Ha-1997}). 

\begin{proposition}  
\label{prop:complete-gen-reg-cyc-mod} 
Consider a regular pair $(q,n)$ and a cyclotomic module 
${\mathcal C}_k \subseteq \fqn$ for some $k\mid n'$. If ${\mathcal C}_k$ is non-exceptional, then $w$ is a complete generator of ${\mathcal C}_k$ over $\fq$ if and only if 
the $q^{\tau_k}$-order of $w$ is equal to $\Phi_{k/\tau_k}(x)^{p^a}$.  
If however ${\mathcal C}_k$ is exceptional, then $w$ is a complete generator of ${\mathcal C}_k$ over $\fq$ if and only if 
the $q^{\tau_k}$-order of $w$ is equal to $\Phi_{k/\tau_k}(x)^{p^a}$ and the 
$q^{2\tau_k}$-order of $w$ is equal to $\Phi_{k/(2\tau_k)}(x)^{p^a}$.   \qed 
\end{proposition} 

\noindent 
In fact, any exceptional cyclotomic module ${\mathcal C}_k$ contains elements which either have $q^{\tau_k}$-order $\Phi_{k/\tau_k}(x)^{p^a}$ 
or $q^{2\tau_k}$-order $\Phi_{k/(2\tau_k)}(x)^{p^a}$ (see Remark \ref{rem:mixed-orders}).  

According to Proposition \ref{prop:complete-gen-reg-cyc-mod}, recalling the notation $\mathcal N$ and $\mathcal E$ from the previous section, we 
define the sets $F_k$ and $F_k^\varepsilon$ as follows. 
\begin{itemize} 
\item For any $k\in {\mathcal N}$, let $F_k$ be the set of monic divisors $f$ of $\Phi_{k/\tau_k}(x)$ such that $f\in \ft{\tau_k}[x]$ and $f$ is irreducible over 
$\ft{\tau_k}$. 
\item 
For any $k\in {\mathcal E}$, let 
$F^\varepsilon_k$ be the set of monic divisors $f$ of $\Phi_{k/(2\tau_k)}(x)$ such that $f\in \ft{\tau_k}[x]$ and $f$ is irreducible over 
$\ft{\tau_k}$. 
\end{itemize}

\noindent 
At this stage we are able to provide a refinement of Proposition \ref{prop:cn-by-can-dec} as follows, where our focus is on the 
exceptional cyclotomic modules. 
For $k\in {\mathcal E}$ we have 
\begin{equation} 
\label{eqn:3D2} {\mathcal C}_k = \bigoplus_{f \in F^\varepsilon_k} W_{k,f}, \end{equation} 
where 
$W_{k,f}$ is the $\ft{\tau_k}[x]$-submodule of ${\mathcal C}_k$ that is annihilated by $f(x^2)^{p^a}$; at the same time, 
$W_{k,f}$ is the $\ft{2\tau_k}[x]$-submodule of ${\mathcal C}_k$ which is annihilated by $f(x)^{p^a}$. 
According to (\ref{eqn:3D2}), 
any $w \in {\mathcal C}_k$ can uniquely be written as $\sum_{f \in F^\varepsilon_k} w_f$, where $w_f \in W_{k,f}$ for every $f$. 
The second part of Proposition \ref{prop:complete-gen-reg-cyc-mod} implies that 
$w$ is a complete generator of ${\mathcal C}_k$ if and only if 
for every $f\in F^\varepsilon_k$ the $q^{\tau_k}$-order of $w_f$ is equal to $f(x^2)^{p^a}$ and the 
$q^{2\tau_k}$-order of $w_f$ is equal to $f(x)^{p^a}$.  
In that case, $w_f$ is called a \textit{complete generator for $W_{k,f}$} over $\fq$. 

If finally we let  
\begin{equation} \label{eqn-Delta-eps} 
\Delta^\varepsilon:=\{(k,f): k \in {\mathcal E} {\mbox{ and }} f \in F^\varepsilon_k\},  
\end{equation} 
the decomposition (\ref{eqn:2B2}) can be refined to 
\begin{equation} 
\label{eqn:4C1} 
\fqn=\left( \bigoplus_{k\in {\mathcal N}} {\mathcal C}_k \right) \oplus 
\left( \bigoplus_{(k,f) \in \Delta^\varepsilon} W_{k,f} \right). 
\end{equation} 
According to this, any $z\in \fqn$ is uniquely decomposed as 
\begin{equation} 
\label{eqn:4C2}z=\sum_{k\in {\mathcal N}} z_{k} + \sum_{k\in \Delta^\varepsilon} z^\varepsilon_{k,f}. \end{equation} 


\section{A character based description of completely normal elements in regular extensions}  
\label{section-05} 
The aim of this section is to efficiently describe the characteristic function of the set of all 
completely normal elements of a regular extension by means of additive finite field characters. 
For the basic theory of characters, see \cite[Chaper 5]{Li-Ni-1986} and Jungnickel \cite[Chapter 7]{Ju-1993}. 

We may start with an arbitrary pair $(q,n)$. 
For simplicity, we write $E=\fqn$. The character group of $(E,+)$, denoted by $\widehat{E}$, is the set 
of all group homomorphisms 
$\chi:(E,+) \rightarrow (\C^*,\cdot)$, where $(\C^*,\cdot)$ is the multiplicative group of the complex numbers. 
Equipped with pointwise multiplication, $\widehat{E}$ is a group which is isomorphic to 
$(E,+)$. The neutral element of $\widehat{E}$ is the trivial additive character, denoted by $\chi_0$. 

Recall from Section \ref{section-02} that $E$ carries several module structures, arising from the intermediate fields over $\fq$. 
Given a divisor $d$ of $n$ and defining 
\begin{equation} 
\label{eqn:4A4} [f(x) \cdot \chi] (z) := \chi(f(\sigma^d)(z)) \ \ {\mbox{(for $z\in E$ and $f(x)\in \fqd[x]$ and $\chi \in \widehat{E}$)}} 
\end{equation} 
shows that $\widehat{E}$ likewise admits the structure of an $\fqd[x]$-module. 
In fact, $\widehat{E}$ and $E$ are even isomorphic as $\fqd[x]$-modules (for any $d\mid n$). 
Therefore, the whole structure and notion of various generators takes over from $E$ (over $\fq$) to the group of additive characters 
$\widehat{E}$ considered as an $\fq$-vector space. 
In particular, the 
$q^d$-order of any $\chi \in \widehat{E}$ (denoted by $\Ord_{q^d}(\chi)$) is the monic polynomial $g\in \fqd[x]$ of least degree such that 
$g(x) \cdot \chi = \chi_0$.

Next, for a divisor $d$ of the $p$-free part $n'$ of $n$ and for a monic polynomial $g \in \fqd[x]$ that divides $x^{n'/d}-1$, we define 
\begin{equation} \label{eqn:Gamma-d-g}
\Gamma_{d,g}:=\{\chi \in \widehat{E}: \Ord_{q^d}(\chi) {\mbox{ divides }} g(x)\}. 
\end{equation}  
This is the $\fqd[x]$-submodule of $\widehat{E}$ which is annihilated by $g(x)$; its cardinality is $q^{d \cdot \deg(g)}$, 
where $\deg(g)$ is the degree of $g$. 
Furthermore, let 
\begin{equation} 
\Gamma_{d,g}^\perp:=\{z \in E: \chi(z)=1 \ {\mbox{ for all $\chi$ in $\Gamma_{d,g}$}}\} 
\end{equation}  
denote the $\fqd[x]$-submodule of $E$ which is dual to $\Gamma_{d,g}$. 
By a basic fact from the theory of characters (see for instance \cite[Lemma 7.1.3]{Ju-1993}), one has 
\begin{equation} \label{eqn:basic-chi}
\sum_{\chi \in \Gamma_{d,g}} \chi(z) 
= \left\{ \begin{array}{cl} 
|\Gamma_{d,g}|,& {\mbox{ if $z\in \Gamma_{d,g}^\perp$,}} \\ 
0,& {\mbox{ if $z\not\in \Gamma_{d,g}^\perp$.}}\end{array} \right. 
\end{equation}  

\noindent 
 Moreover, with $m:=n'/d$ and $\widehat{g}(x):=(x^{mp^a}-1)/g(x)$ being the cofactor of $g$ of 
 the minimal polynomial $x^{n/d}-1$ of $E$ (with respect to $\sigma^d$), 
we write 
\begin{equation}  \label{M-module} 
M_{d,\widehat{g}}:=\{w \in E: \widehat{g}(\sigma^d)(w)=0\} \end{equation}  
for the $\fqd[x]$-submodule of $E$ which is annihilated by $\widehat{g}(x)$. 
Again by the theory of characters, one has $\Gamma_{d,g}^\perp = M_{d,\widehat{g}}$, and therefore, 
altogether, 
$q^{-d\cdot \deg(g)} \cdot \sum_{\chi \in \Gamma_{d,g}} \chi$ 
is the characteristic function of the set of all elements of $E$ that belong to $M_{d,\widehat{g}}$.

Now, let $\phi_{q^d}$ and $\mu_{q^d}$, respectively denote the Euler- and the M\"obius function for the ring 
$\fqd[x]$. 
Since $n'$ is relatively prime to $p$, the irreducible $\fqd[x]$-factors of $g$ occur with multiplicity $1$. 
Consequently, $\mu_{q^d}(g):=(-1)^{i(g)}$, where $i(g)$ is the number of distinct monic factors of $g$ that are irreducible 
over $\fqd$, while $\phi_{q^d}(g)$ is the number of units of the residue ring $\fqd[x]/(g)$. 
Finally, let  
\begin{equation} 
\label{eqn:4B1} 
A^g_d:=\frac{\phi_{q^d}(g)}{q^{d\cdot \deg(g)}} \sum_{\chi \in \Gamma_{d,g}} \frac{\mu_{q^d}(\Ord_{q^d}(\chi))}{\phi_{q^d}(\Ord_{q^d}(\chi))} 
\cdot \chi, \end{equation} 
an element of the $\C$-vector space $\C^E$ of all mappings from $E$ to $\C$. 
The important feature about $A^g_d$ is the following.   

\begin{proposition} \label{prop:Agd} 
$A^g_d$ is the characteristic function of the set of those $z\in E=\fqn$ whose $q^d$-order is divisible by $g(x)^{p^a}$. 
\end{proposition} 

\begin{proof} Consider the factorization of 
$g(x)$ into monic factors that are irreducible over $\fqd$, say $g(x):=\prod_{i=1}^t h_i(x)$.  
Because of the multiplicativity of the M\"obius function and that of the Euler function, since  
$q^{d \cdot \deg(g)}=\prod_{i=1}^t q^{d \cdot \deg(h_i)}$ and as $\Gamma_{d,g}$ in $\widehat{E}$ decomposes into the direct product 
of the $\fqd[x]$-submodules $\Gamma_{d,h_i}$, 
we obtain the multiplicativity of the functions as in 
(\ref{eqn:4B1}); this means   
\begin{equation} \label{eqn:multAhd} A^g_d = \prod_{i=1}^t A^{h_i}_d. 
\end{equation}  
Suppose now, that $h(x)\in \fqd[x]$ is some monic divisor of $g(x)$ which is irreducible over $\fqd$. 
Then $\mu_{q^d}(h)=-1$, and the fact that $\Ord_{q^d}(\chi)=h(x)$ for every nontrivial character $\chi$ of $\Gamma_{d,h}$ 
gives (after some simplifications) 
\begin{equation} \label{eqn:formula-Ahd} 
q^{d \cdot\deg(h)} \cdot A^h_d = (\phi_{q^d}(h)+1) \cdot \chi_0 - \sum_{\chi \in \Gamma_{d,h}} \chi. 
\end{equation}  
As mentioned before, 
$\Gamma_{d,h}^\perp = M_{d,\widehat{h}}$, and 
$q^{-d\cdot \deg(h)} \cdot \sum_{\chi \in \Gamma_{d,h}} \chi$ 
is the characteristic function of the set of elements that belong to $M_{d,\widehat{h}}$. 
Now, back to the formula (\ref{eqn:formula-Ahd}), and observing that $\phi_{q^d}(h)=q^{d\cdot \deg(h)}-1$ (as $h(x)$ is irreducible) we achive that 
$A^h_d(w)=1$ if $w$ is not a member of $M_{d,\widehat{h}}$, and $A^h_d(w)=0$, else. 
But $w\not\in M_{d,\widehat{h}}$ means that $w$ is not annihilated by $\widehat{h}(x)$, and this is equivalent to the fact that 
$h(x)^{p^a}$ divides the $q^d$-order of $w$. 
Consequently, because of (\ref{eqn:multAhd}), $A^g_d(w)=1$ if and only if $h_i(x)^{p^a}$ divides the $q^d$-order of $w$ for any $i$, 
that is, if and only if $g(x)^{p^a}$ divides $\Ord_{q^d}(w)$, and $A^g_d(w)=0$, else. 
\end{proof}

\noindent 
We are now returning to the decompositions in (\ref{eqn:4C1}) and (\ref{eqn:4C2}). 
For $k\in {\mathcal N}$ one has  
$A^{\Phi_{k/\tau_k}}_{\tau_k}(z)=1$ if and only if the $q^{\tau_k}$-order of $z$ is divisible 
by $\Phi_{k/\tau_k}(x)^{p^a}$, and this holds if and only if $\Ord_{q^{\tau_k}}(z_{k})$ is equal to $\Phi_{k/\tau_k}(x)^{p^a}$. 
Analogously, with $(k,f) \in \Delta^\varepsilon$ we have  
$A^{f(x^2)}_{\tau_k}(z) \cdot A^{f}_{2\tau_k}(z)=1$ if and only if the $(k,f)$-component $z^\varepsilon_{k,f}$ of $z$ 
has $q^{\tau_k}$-order $f(x^2)^{p^a}$ and $q^{2\tau_k}$-order $f(x)^{p^a}$. 
Considering $\C^E$ once more as $\C$-algebra equipped with the 
pointwise multiplication of functions, we therefore altogether obtain  

\begin{proposition} \label{prop:character-pcn} 
Let $\fqn$ be a regular extension over $\fq$. Then the characteristic function of the set of all elements of $\fqn$ that are completely normal over $\fq$ 
is equal to  
\[A^c:= \left(\prod_{k\in {\mathcal N}} A^{\Phi_{k/\tau_k}}_{\tau_k} \right) \cdot 
\left(\prod_{(k,f) \in \Delta^\varepsilon} [A^{f(x^2)}_{\tau_k}A^{f}_{2\tau_k}] \right). \]  
\qed 
\end{proposition} 

\begin{remark} \label{rem:simple} 
In Section \ref{section-07}, after additionally considering the primitivity condition, 
it will be convenient to adjust the notation as follows. 
For an index $k\in  {\mathcal N}$, we first let  
\begin{equation} \label{eqn:Afd-simple} 
B_{k} := \frac{q^{\tau_k\cdot \deg(\Phi_{k/\tau_k})}}{\phi_{q^{\tau_k}}(\Phi_{k/\tau_k})} \cdot  A^{\Phi_{k/\tau_k}}_{\tau_k}. 
\end{equation} 
Now, 
\begin{itemize}
\item we simply write 
$\phi_{k}$ for the Euler function $\phi_{q^{\tau_k}}$, 
\item as well as $\mu_{k}$ for the M\"obius function $\mu_{q^{\tau_k}}$. 
\item Furthermore, we write 
$\Ord_k$ instead of $\Ord_{q^{\tau_k}}$, 
\item and we abbreviate $\Gamma_{\tau_k,\Phi_{k/\tau_k}}$ to $\Gamma_k$. 
\end{itemize} 
Then altogether we obtain 
\begin{equation} \label{eqn:Bk-simple} 
B_k = \sum_{\chi \in \Gamma_k} \frac{\mu_k(\Ord_k(\chi))}{\phi_k(\Ord_k(\chi))} \cdot \chi. 
\end{equation} 
An appropriate notation for indices $(k,f) \in \Delta^\varepsilon$ 
are proposed at the end of the next section. \qed 
\end{remark}

\section{An effective character theoretical description for exceptional cyclotomic modules}   
\label{section-06} 
We have now arrived at the heart of the problem. 
The aim of the present section is to present an effective description 
of the product $A^{f(x^2)}_{\tau_k}A^{f}_{2\tau_k}$, where $k$ is from  
the index set ${\mathcal E}$ of exceptional cyclotomic modules, 
and where $f(x)$ is a monic divisor of $\Phi_{k/(2\tau_k)}(x)$ which is irreducible over 
$\mathbb{F}_{q^{\tau_k}}$. 
In order to keep the terminology as simple as possible, we presently write $Q:=q^{\tau_k}$ and let $K:=\mathbb{F}_{Q}$ and $L:=\mathbb{F}_{Q^2}$. 
Furthermore, let $S=\sigma^{\tau_k}$ be the Frobenius-automorphism of $\fqn/K$. 

\begin{enumerate} 
\item As $S$-invariant $K$-vector space, $W_{k,f}$ is annihilated by 
$f(x^2)^{p^a}$. Since $\tau_k$ is odd and $k$ is divisible by $8$ we have that 
$\Phi_{k/(2\tau_k)}(x^2)=\Phi_{k/\tau_k}(x)$, and therefore $f(x^2)$ is a divisor of $\Phi_{k/\tau_k}(x)$. 
Over $K[x]$, the polynomial $f(x^2)$ splits into two irreducible divisors (of equal degree), say $g_1(x)$ and $g_2(x)$, and therefore, 
as a $K[x]$-module (with respect to $S$), $W_{k,f}$ decomposes into 
\begin{equation} \label{eqn:Wkf-as-K} 
W_{k,f}=M_{\tau_k,f(x^2)^{p^a}} = M_{\tau_k,g_1(x)^{p^a}} \oplus M_{\tau_k,g_2(x)^{p^a}} 
\end{equation}  
and, for $w\in W_{k,f}$, we write $w=u_1+u_2$ according to this decomposition. 
(Remember the notion `$M$' for certain submodules of $E$ in (\ref{M-module}).) 
\item 
When considering $W_{k,f}$ as an $S^2$-invariant $L$-vector space, we use the indeterminate $y$ instead of $x$. 
Over $L[y]$, the polynomial $f(y)$ splits into two irreducible divisors (of equal degree), say $h_1(y)$ and $h_2(y)$, and therefore we obtain 
\begin{equation} \label{eqn:Wkf-as-L} W_{k,f}=M_{2\tau_k,f(y)^{p^a}}= M_{2\tau_k,h_1(y)^{p^a}} \oplus M_{2\tau_k,h_2(y)^{p^a}}. 
\end{equation} 
According to this, any $w\in W_{k,f}$ is decomposed as $w=v_1+v_2$. 
 \end{enumerate} 

\begin{remark} \label{rem:mixed-orders} 
The fundamental feature concerning the complete structure of $W_{k,f}$ is that any 
$w \in W^k_f$ with $\Ord_{Q^2}(w)=h_1(y)^{p^a}h_2(y)^\beta$ or $\Ord_{Q^2}(w)=h_1(y)^\alpha h_2(y)^{p^a}$ where $\alpha, \beta < p^a$
 has $Q$-order $f(x^2)^{p^a}$. Symmetrically, 
every $w \in W_{k,f}$ with $\Ord_{Q}(w)=g_1(x)^{p^a}g_2(x)^{\beta}$ or $\Ord_{Q}(w)=g_1(x)^{\alpha}g_2(x)^{p^a}$ and $\alpha, \beta < p^a$ 
has $Q^2$-order $f(y)^{p^a}$. This is a crucial fact which has been conjectured in \cite[p. 125]{Ha-1997}. 
Because of its importance we have to include a proof, which however is postponed to the last section. 
With this information at hand, we can recover the number of all those $w \in W_{k,f}$ whose order-pair 
$(\Ord_Q(w),\Ord_{Q^2}(w))$ is equal to $(f(x^2)^{p^a},f(y)^{p^a})$, and then may altogether 
count the number of completely normal elements in any regular extension.  \qed 
\end{remark} 

\noindent 
A similar situation as outlined in Remark \ref{rem:mixed-orders} occurs within the character group $\widehat{E}$ of the additive group of $E=\fqn$, and 
we are going to describe this in detail, next. 
As the M\"obius functions (occuring in the definition of the functions $A^g_d$ in (\ref{eqn:4B1})) 
vanish on polynomials which are divisible by the square 
of an irreducible polynomial, we can restrict our attention to the polynomial-pair $(f(x^2),f(y))$, getting rid of  
the power $p^a$. 
For $k \in {\mathcal E}$ and $f\in F_k^\varepsilon$ as above (that is, for $(k,f) \in \Delta^\varepsilon$), recalling the notion in 
(\ref{eqn:Gamma-d-g}), we write 
\begin{equation} \label{eqn:Gamma-k-f}
\Gamma^\varepsilon_{k,f} := \Gamma_{\tau_k,f(x^2)} = \Gamma_{2\tau_k,f(y)} 
\end{equation} 
for the set of all characters $\chi \in \widehat{E}$ which 
(with respect to $S$) are annihilated by $f(x^2)\in K[x]$; this is likewise the 
set of all characters which 
(with respect to $S^2$) are annihilated by $f(y)\in L[y]$.  
As a $K[x]$-module, $\Gamma^\varepsilon_{k,f}$ decomposes into $\Gamma_{\tau_k,g_1}$ and $\Gamma_{\tau_k,g_2}$, 
and as an $L[y]$-module, $\Gamma^\varepsilon_{k,f}$ decomposes into $\Gamma_{2\tau_k,h_1}$ and $\Gamma_{2\tau_k,h_2}$. 

Essentially, 
since the $\Gamma_{\tau_k,g_i}$ are not invariant under the multiplication with $L$ and since 
the $\Gamma_{2\tau_k,h_j}$ are not invariant under the action of $S$, these $K$-subspaces have pairwise trivial intersection. 
(This argument is worked out in Section \ref{section-13} for the situation described in Remark \ref{rem:mixed-orders}.)
Together with the trivial $K$-subspaces of $\Gamma^\varepsilon_{k,f}$ we obtain a lattice 
of six $K$-subspaces, ordered by the inclusion of sets. 

As a consequence of this discussion, if $\Ord_Q(\chi)=g_i(x)$ for some $i=1,2$, then $\Ord_{Q^2}(\chi)=f(y)$, while $\Ord_{Q^2}(\chi)=h_j(y)$ for some $j=1,2$ implies $\Ord_Q(\chi)=f(x^2)$. 
Consequently, there are the six possible pairs of orders $(\Ord_{Q}(\chi),\Ord_{Q^2}(\chi))$ given in table (\ref{tab:1}). 
With respect to componentwise divisibility, these elements build a lattice ${\mathcal L}={\mathcal L}^\varepsilon_{k,f}$ with least element $(1,1)$ and 
maximum $(f(x^2),f(y))$, while 
the four other pairs are atoms. 

Of course, this lattice corresponds to the lattice of the six $K$-subspaces mentioned above. 
Consequently, the M\"obius function of $\mathcal L$, denoted by $\mu^\varepsilon_{k,f}$ is as given in table (\ref{tab:1}). 
The lattice ${\mathcal L}$ also admits an Euler function, denoted by  $\phi^\varepsilon_{k,f}$: 
For every $\ell \in {\mathcal L}$, the term $\phi^\varepsilon_{k,f}(\ell)$ is defined to be the number of characters $\chi$ such that 
$(\Ord_{Q}(\chi),\Ord_{Q^2}(\chi))=\ell$. For simplicity, we write 
$\Ord^\varepsilon_{k,f}(\chi)$ for this pair of orders. With $\delta$ being the degree of $f(x)$, 
and from what has been said above, we obtain the following values:

\begin{equation} \label{tab:1} 
\begin{array}{c | c | c} 
{\mbox{order-pair $\ell$}} & {\mbox{$\mu^\varepsilon_{k,f}$-value}} & {\mbox{$\phi^\varepsilon_{k,f}$-value}} \\ \hline  
(1,1) & 1 & 1 \\ 
(f(x^2),h_1(y)) & -1 & Q^\delta-1 \\  
(f(x^2),h_2(y)) & -1 & Q^\delta-1 \\  
(g_1(x),f(y)) & -1 & Q^\delta-1 \\  
(g_2(x),f(y)) & -1 & Q^\delta-1 \\  
(f(x^2),f(y)) & 3 & Q^{2\delta}-4Q^\delta +3\end{array} \end{equation} 

\noindent 
We remark that $\delta=\ord_{k/(2\tau_k)}(Q)=\ord_k(q)/\tau_k^2$ (see also Section \ref{section-08}). 
The total number of elements of $\Gamma^\varepsilon_{k,f}$ is $Q^{2\delta}$. This altogether leads us to the following result.

\begin{proposition} \label{prop:charAA} For every $k\in {\mathcal E}$ and every $f\in F_k^\varepsilon$, the characteristic function of the set of all those elements of $\fqn$ whose $(k,f)$-component is a complete generator of the 
module $W_{k,f}$ is equal to 
\begin{equation} \label{eqn:new-description} A^{f(x^2)}_{\tau_k} A^{f}_{2\tau_k} = \frac{Q^{2\delta}-4Q^\delta +3}{Q^{2\delta}} \sum_{\ell \in {\mathcal L}} 
\frac{\mu^\varepsilon_{k,f}(\ell)}{\phi^\varepsilon_{k,f}(\ell)} \sum_{\chi :\ell} \chi. 
\end{equation}  
In that formula, $Q=q^{\tau_k}$ and $\delta=\deg(f)$ and ${\mathcal L}={\mathcal L}^\varepsilon_{k,f}$. 
Moreover, the 
second sum, indexed by $\chi:\ell$  runs over all $\chi \in \Gamma^\varepsilon_{k,f}$ with 
order pair $\Ord^\varepsilon_{k,f}(\chi)=(\Ord_Q(\chi),\Ord_{Q^2}(\chi))$ being equal to $\ell$. 
\end{proposition} 

\begin{proof} 
For an element  $z\in E=\fqn$ let $w:=z^\varepsilon_{k,f} \in W_{k,f}$ be its $(k,f)$-component. 
From the discussion at the beginning of this section, see (\ref{eqn:Wkf-as-K}) and (\ref{eqn:Wkf-as-L}), 
we write $w=u_1+u_2$ as well as $w=v_1+v_2$. 
We shall occasionally omit the variable names `$x$' and `$y$' if clear from the context, and 
are now going to examine the sum $\sum_{\chi :\ell} \chi$ for all possibilities of $\ell \in {\mathcal L}={\mathcal L}^\varepsilon_{k,f}$. 
For simplicity, we write $u_i\equiv 0$ if $\Ord_{Q}(u_i)\mid g_i^{p^a-1}$ and $u_i \not\equiv 0$, else. 
Similarly, write $v_j\equiv 0$ if $\Ord_{Q^2}(v_j)\mid h_j^{p^a-1}$ and $v_j \not\equiv 0$, else. 
\begin{enumerate} 
\item[(i)] If $\ell=(1,1)$, then $\sum_{\chi :\ell} \chi(w)=\chi_0(w)=1$. 
\item[(ii)] Assume next that $\ell=(g_i,f)$ for some $i=1,2$ and let $X_i$ be the sum over all 
characters $\chi$ with $\Ord_Q(\chi)=g_i$. Using the corresponding 
basic evaluation of character sums as in (\ref{eqn:basic-chi}), we obtain  
\[X_i(w)=\sum_{\chi :\ell} \chi(w) = \sum_{\chi \in \Gamma_{\tau_k,g_i}} \chi(w)-1 = \left\{ 
\begin{array}{cl} 
-1, & {\mbox{ if }} u_i \not\equiv 0 \\ 
Q^\delta-1,& {\mbox{ if }} u_i \equiv 0 .\end{array}\right.\] 
\item[(iii)] If $\ell=(f(x^2),h_j)$ for some $j=1,2$, then, similarly, with 
$Y_j:=\sum_{\chi :\ell} \chi$, we have that $Y_j$ is the sum over all characters with $Q^2$-order equal to $h_j$, and therefore 
\[Y_j(w)=\sum_{\chi :\ell} \chi(w) = \sum_{\chi \in \Gamma_{2\tau_k,h_j}} \chi(w)-1 = \left\{ 
\begin{array}{cl} 
-1, & {\mbox{ if }} v_j \not\equiv 0  \\ 
Q^{\delta}-1,& {\mbox{ if }} v_j \equiv 0 .\end{array}\right.\] 
\item[(iv)] Finally, let $\ell=(f(x^2),f(y))$. 
For short, let $Z:=\sum_{\chi :\ell}\chi$. 
Then 
$$Z(w)=\sum_{\chi \in \Gamma^\varepsilon_{k,f}} \chi(w)-X_1(w)-X_2(w)-Y_1(w)-Y_2(w)-1.$$ 
If $u_1\equiv 0$ and $u_2\equiv 0$, then  $v_1 \equiv 0$ and $v_2\equiv 0$ (and vice versa), and therefore  
$Z(w)=Q^{2\delta} - 4(Q^\delta-1)-1 = Q^{2\delta} - 4Q^\delta +3$.  
If $u_1\equiv 0$ but $u_2\not\equiv 0$, then $v_1 \not\equiv 0$ and $v_2\not\equiv 0$ (by what has been said in Remark \ref{rem:mixed-orders}), and 
therefore 
$Z(w)=0 - (Q^\delta-1) - 3 \cdot (-1) -1 = -(Q^\delta-3)$. 
We also obtain $Z(w)=-(Q^\delta-3)$ provided that 
$(\Ord_Q(w),\Ord_{Q^2}(w))$ is some other atom in the lattice $\mathcal L$ (which means that exactly one of $u_1$, $u_2$, $v_1$, $v_2$ is $\equiv 0$).  
If finally $u_i\not\equiv 0$ for $i=1,2$ and $v_j\not\equiv 0$ for $j=1,2$, then 
$Z(w)=0 - 4\cdot (-1)-1=3$.  
\end{enumerate} 
Now, disregarding the normalizing factor, the right hand side of the formula (\ref{eqn:new-description}) evaluated at $z$ is the same as its evaluation at $w$, 
namely 
\begin{equation} \label{eqn:evaluate} \sum_{\ell \in {\mathcal L}} 
\frac{\mu^\varepsilon_{k,f}(\ell)}{\phi^\varepsilon_{k,f}(\ell)} \sum_{\chi :\ell} \chi(z) 
= \frac{3\cdot Z(w)}{Q^{2\delta}-4Q^\delta+3}  - 
\frac{X_1(w)+X_2(w)+Y_1(w)+Y_2(w)}{Q^\delta-1} +1.\end{equation} 
With the discussion in (i)-(iv) above we determine that this gives $0$ if $w$ is \textit{not} a complete generator of $W_{k,f}$, 
that is, at least one of $u_1$, $u_2$, respectively $v_1$, $v_2$ has not the maximal $Q$-, respectively $Q^2$-order. 
On the other hand, if $w$ is a complete generator of $W_{k,f}$, then  
(\ref{eqn:evaluate}) reduces to   
\[\frac{9}{Q^{2\delta}-4Q^\delta+3} - \frac{-4}{Q^\delta-1} +1 
= \frac{9+4 \cdot (Q^\delta-3) + Q^{2\delta} -4Q^\delta +3}{Q^{2\delta}-4Q^\delta+3} 
= \frac{Q^{2\delta}}{Q^{2\delta}-4Q^\delta+3},\] 
and this altogether establishes the proof of Proposition \ref{prop:charAA}. 
\end{proof} 

\begin{remark} \label{rem:simple-2}  
Similar to Remark \ref{rem:simple}, for a pair $(k,f)\in \Delta^\varepsilon_k$, we introduce the 
following simpler terminology:  Let 
\begin{equation} \label{eqn:Afd-simple-e} 
B^\varepsilon_{k,f} := \frac{Q^{2\delta}}{Q^{2\delta}-4Q^\delta +3} \cdot A^{f(x^2)}_{\tau_k} A^{f}_{2\tau_k} .
\end{equation} 
Then, from Proposition \ref{prop:charAA} and its proof we have 
\begin{equation} 
B^\varepsilon_{k,f} = 
\sum_{\chi \in \Gamma^\varepsilon_{k,f}} \frac{\mu^\varepsilon_{k,f}(\Ord^\varepsilon_{k,f}(\chi))}
{\phi^\varepsilon_{k,f}(\Ord^\varepsilon_{k,f}(\chi))} \cdot \chi,  
\end{equation} 
where $\Gamma^\varepsilon_{k,f}$ is as in (\ref{eqn:Gamma-k-f}). 
In view of the proof of the forthcoming Proposition \ref{prop:suff-crit}, we finally define 
\begin{equation} \label{eqn:Bc}
B^c := \left(\prod_{k\in {\mathcal N}} B_k \right) \cdot 
\left(\prod_{(k,f)\in \Delta^\varepsilon} B^\varepsilon_{k,f} \right). 
\end{equation} \qed 
\end{remark} 

\section{A sufficient number theoretical condition}
\label{section-07} 
In the present section we are going to present a sufficient number theoretical condition for the existence of 
primitive complete normal bases in a regular extension $\fqn/\fq$. So, for the first time, we are confronted with the primitivity condition. 
Recalling the definitions of $F_k$ (for $k\in {\mathcal N}$) and $F_k^\varepsilon$ (for $k\in {\mathcal E}$) from Section \ref{section-04}, 
we let 
\begin{equation} \label{eqn:Omega-eps} \Omega:=\sum_{k\in {\mathcal N}} |F_k| \ {\mbox{ and }} \ \Omega^\varepsilon:=\sum_{k\in {\mathcal E}} |F^\varepsilon_k| 
\end{equation}  
(with the convention that $\Omega^\varepsilon=0$ if ${\mathcal E}$ is empty). 

\begin{proposition} \label{prop:suff-crit}
Assume that $\fqn$ is a regular extension over $\fq$, where $q\equiv 3\MOD 4$ and $n$ is even. 
Let $\Omega^c:=\Omega + 3\Omega^\varepsilon$. Furthermore, as for the multiplicative part, let $\omega$ be the number of distinct prime divisors of $q^n-1$. 
Suppose that 
\[\sqrt{q^n} > (2^\omega-1) \cdot (2^{\Omega^c}-1).\] 
Then there exists a primitive element in $\fqn$ which  is completely normal over $\fq$. 
\end{proposition} 

\begin{proof} We first recall a description of the characteristic function of the set of all primitive elements of 
$E:=\fqn$ (see \cite[Section 7.5]{Ju-1993}, for instance). The group of multiplicative characters of $E$, denoted by 
$\widehat{E^*}$, is the set 
of all group homomorphisms 
$\psi:(E^*,\cdot) \rightarrow (\C^*,\cdot)$, equipped with pointwise multiplication. 
The neutral element of $\widehat{E^*}$ is the trivial multiplicative character, denoted by $\psi_0$. 
This group is isomorphic to 
the multiplicative group of $E$, hence cyclic of order $q^n-1$. 
Therefore, the notion of the \textit{multiplicative order} is also used for elements of $\widehat{E^*}$: 
If $u\in E^*$, then $\ord(u)$ is the number of elements of the subgroup of $E^*$ which is generated by $u$, 
and analogously, for any multiplicative character $\psi$, the minimal integer $\ell \geq 1$ such that $\psi^\ell=\psi_0$ 
is denoted as $\ord(\psi)$. 
It is convenient to extend the domain of any multiplicative character to the whole of $E$ by defining $\psi(0):=0$ if $\psi\not=\psi_0$, while 
$\psi_0(0):=1$. 

Letting $\mu$ denote the M\"obius function on the ring of integers, then 
\begin{equation} \label{eqn:P} 
P := \frac{\varphi(q^n-1)}{q^n-1} \sum_{\psi\in \widehat{E^*}} \frac{\mu(\ord(\psi))}{\varphi(\ord(\psi))} \cdot \psi 
\end{equation} 
is the characteristic function of the set of all primitive elements of $E$. 
Using properties of the M\"obius function and the distribution of orders of the multiplicative characters, one may also write 
\begin{equation} \label{eqn:P-alt} 
P := \frac{\varphi(q^n-1)}{q^n-1} \sum_{e|\rad(q^n-1)} \frac{\mu(e)}{\varphi(e)} \sum_{\psi :e} \psi,  
\end{equation} 
where the summation index $\psi :e$ means that the sum runs over all $\varphi(e)$ multiplicative characters $\psi$ with order $e$. 
(As mentioned earlier, the function $\rad$ gives the radical of its argument.)  

Consequently, from Proposition \ref{prop:character-pcn}, we obtain that $PA^c$ is the characteristic function  of 
the set of all primitive completely normal elements of $E/F$ where $F:=\fq$. 
We want to derive a sufficient condition for 
$\sum_{w\in E} P(w)A^c(w)$ to be non-zero. So, by contraposition, assume that there is no primitive completely normal 
element in $E/F$. Then the latter sum is equal to zero, and therefore as well, 
\[ \frac{q^n-1}{\varphi(q^n-1)} \cdot \sum_{w \in E}   P(w) B^c(w) = 0, \] 
where $B^c$ is as in (\ref{eqn:Bc}). 
The relevant part of the additive character group is 
\begin{equation} \label{eqn:hat-Gamma} 
\widehat{\Gamma} := \{ \chi \in \widehat{E}: \Ord_q(\chi) {\mbox{ divides }} x^{n'}-1\}. 
\end{equation}  
With the notation from the end of Section \ref{section-05} and from (\ref{eqn:Gamma-k-f}), since character groups are written multiplicatively, 
$\widehat{\Gamma}$ is directly decomposed into 
\begin{equation} \label{eqn:hat-Gamma-dec} \widehat{\Gamma} = \left( \prod_{k \in {\mathcal N}} \Gamma_k \right) \cdot 
\left( \prod_{(k,f)\in \Delta^\varepsilon} \Gamma^\varepsilon_{k,f} \right). \end{equation}
According to this, any character $\chi$ from $\widehat{\Gamma}$ is decomposed as 
\begin{equation} \label{eqn:chi-dec} 
 \chi  = \left( \prod_{k \in {\mathcal N}} \chi_k \right) \cdot 
\left( \prod_{(k,f)\in \Delta^\varepsilon} \chi^\varepsilon_{k,f} \right). \end{equation} 
Furthermore, we write  $\Ord^c(\chi)$ for the tuple of respective orders of the components of $\chi$, that is, 
\begin{equation} \label{eqn:ord-dec}  \Ord^c(\chi) = \left( \mathop{\times}_{k \in {\mathcal N}} \Ord_k(\chi_k) \right) \times 
\left( \mathop{\times}_{(k,f)\in \Delta^\varepsilon} \Ord^\varepsilon_{k,f}(\chi^\varepsilon_{k,f}) \right). \end{equation}  
In the same way, we deal with the M\"obius functions involved, and expanding them multiplicatively, we define  
\begin{equation} \label{eqn:mu-dec} \mu^c(\Ord^c(\chi)) :=  
\left( \prod_{k \in {\mathcal N}} \mu_{k}(\Ord_{k}(\chi_{k})) \right) \cdot 
\left( \prod_{(k,f)\in \Delta^\varepsilon} \mu^\varepsilon_{k,f}(\Ord^\varepsilon_{k,f}(\chi^\varepsilon_{k,f})) \right). \end{equation}
And, of course, the various Euler functions are composed as well and lead to 
\begin{equation} \label{eqn:phi-dec} \phi^c(\Ord^c(\chi)) :=  
\left( \prod_{k \in {\mathcal N}} \phi_{k}(\Ord_{k}(\chi_{k})) \right) \cdot 
\left( \prod_{(k,f)\in \Delta^\varepsilon} \phi^\varepsilon_{k,f}(\Ord^\varepsilon_{k,f}(\chi^\varepsilon_{k,f})) \right). \end{equation}
This altogether gives us 
\[0 = \sum_{\psi \in \widehat{E^*}} \sum_{\chi \in \widehat{\Gamma}} \frac{\mu(\ord(\psi))}{\varphi(\ord(\psi))}  
\frac{\mu^c(\Ord^c(\chi))}{\phi^c(\Ord(\chi))} \cdot G(\psi,\chi),\]
where 
$G(\psi,\chi)$ is the Gauss sum $\sum_{w \in E} \psi(w) \chi(w)$. 
Now it is well known (\cite{Li-Ni-1986,Ju-1993}) that $G(\psi_0,\chi_0)=q^n$, while $G(\psi,\chi)=0$ if either $\psi=\psi_0$ or $\chi=\chi_0$.  

This implies  
\[-q^n = \sum_{\psi \in \widehat{E^*}, \atop \psi \not=\psi_0} \sum_{\chi \in \widehat{\Gamma}, \atop \chi \not=\chi_0} \frac{\mu(\ord(\psi))}{\varphi(\ord(\psi))}  
\frac{\mu^c(\Ord^c(\chi))}{\phi^c(\Ord(\chi))} \cdot G(\psi,\chi).\]
If $\psi\not=\psi_0$ and $\chi\not=\chi_0$, then the absolute value of $G(\psi,\chi)$ is equal to 
$q^{n/2}$. So, taking absolute values on both sides of the last expression and applying the triangle inequality gives 
\begin{equation} \label{eqn:ineq} q^n \leq 
\sum_{\psi \in \widehat{E^*}, \atop \psi \not=\psi_0} \sum_{\chi \in \widehat{\Gamma}, \atop \chi \not=\chi_0} 
\frac{|\mu(\ord(\psi))|}{\varphi(\ord(\psi))}  
\frac{|\mu^c(\Ord^c(\chi))|}{\phi^c(\Ord(\chi))} \cdot q^{n/2}.\end{equation} 
As mentioned above, the sum over the multiplicative characters only has to run over those $\psi$ with $\ord(\psi)$ dividing 
$\rad(q^n-1)$. Moreover, for a given divisor $e$ of $\rad(q^n-1)$ there are exactly $\varphi(e)$  multiplicative characters with order $e$. 
Similar, on the additive side, for any $\gamma$ of the underlying (complete) lattice 
\begin{equation} \label{eqn:L-dec} {\mathcal L}^c:= \left( \mathop{\times}_{k \in {\mathcal N}} {\mathcal L}_k \right) \times 
\left( \mathop{\times}_{(k,f)\in \Delta^\varepsilon} {\mathcal L}^\varepsilon_{k,f} \right) \end{equation}  
(where for $k \in {\mathcal N}$, the lattice ${\mathcal L}_k$ consists of all monic divisors of $\Phi_{k/\tau_k}(x)$ with coefficients from $\mathbb{F}_{q^{\tau_k}}$), 
there are precisely $\phi^c(\gamma)$ additive characters $\chi$ of $\Gamma$ such that $\phi^c(\Ord^c(\chi))=\phi^c(\gamma)$. 
Let $\gamma_0$ be the element of ${\mathcal L}^c$ having all its components equal to $1$ (for $k \in {\mathcal N}$), respectively $(1,1)$ (when $(k,f)$ is from $\Delta^\varepsilon$). Then $\gamma_0$ is just equal to $\Ord^c(\chi_0)$. 
With this notation at hand, the above inequality (\ref{eqn:ineq}) can be transformed to 
\[ q^{n/2} \leq \left( \sum_{e|\rad(q^n-1), \atop e \not=1} |\mu(e)| \right) \cdot 
\left( \sum_{\gamma \in {\mathcal L}^c, \atop \gamma \not= \gamma_0} |\mu^c(\gamma)|  \right).\] 
Now, the first factor of the right hand side is equal to the number of all divisors of $\rad(q^n-1)$ distinct from $1$, and by the definition of $\omega$ this is equal to $2^\omega-1$. The second factor can be expressed as 
\begin{equation} \label{eqn:factor-mu} 
 -1 + \left( \prod_{k \in {\mathcal N}}  \sum_{\alpha \in {\mathcal L}_k} |\mu_{k}(\alpha)| \right) 
\cdot \left( \prod_{(k,f)\in \Delta^\varepsilon}  \sum_{\beta \in {\mathcal L}^\varepsilon_{k,f}} 
|\mu^\varepsilon_{k,f}(\beta)| \right), \end{equation}   
where $-1$ takes $\gamma_0$ into account. Since $|{\mathcal L}_k|= 2^{|F_k|}$ and $|\mu_k(\alpha)|=1$ (for any $k\in {\mathcal N}$ and any $\alpha \in {\mathcal L}_k$), the first product by (\ref{eqn:Omega-eps}) is equal to 
\[ \prod_{k\in {\mathcal N}} 2^{|F_k|} = 2^{\Omega}. \] 
It remains to consider the second factor, but with the content of Section \ref{section-06} (see in particular (\ref{tab:1})), 
we obtain 
\[\sum_{\beta \in {\mathcal L}_{k,f}} |\mu^\varepsilon_{k,f}(\beta)| = 1 + 4 \cdot |-1| + 3 = 8\] 
for any pair $(k,f) \in \Delta^\varepsilon$, and therefore the second factor in (\ref{eqn:factor-mu}) by the definition of  $\Omega^{\varepsilon}$ 
in (\ref{eqn:Omega-eps}) gives  
\[ 8^{|\Delta^\varepsilon|} = 8^{\Omega^{\varepsilon}} = 2^{3 \cdot \Omega^{\varepsilon}}.\] 
This finally completes the proof of Proposition \ref{prop:suff-crit}. 
\end{proof} 

\noindent 
In the following two sections, we are going to demonstrate the strength of the criterion in Proposition \ref{prop:suff-crit}. 
In fact, for the case where 
$n\equiv 0 \MOD 8$, we will achieve the following result.  

\begin{proposition} \label{prop:evaluation} 
Assume that $\fqn$ is a regular extension over $\fq$, where $q \equiv 3 \MOD 4$ and $n \equiv 0 \MOD 8$. 
Then $q^{n/2} \leq (2^\omega-1) \cdot (2^{\Omega^c}-1)$ if and only if 
$q=3$ and $n\in \{8, 16\}$. 
\end{proposition} 

\noindent 
Observing that the pairs $(q,n)=(3,8)$ and $(q,n)=(3,16)$ are covered by the computational results of Morgan and Mullen 
\cite{Mo-Mu-1996} (see the polynomials in (\ref{3-8-and-3-16})), 
and that $(q,n)=(3,8)$ is additionally covered by the theoretical contribution of Blessenohl \cite{Bl-2005}, 
Proposition \ref{prop:evaluation} essentially implies the assertion of Theorem \ref{thm:main} for $n\equiv 0 \MOD 8$. 
The remaining cases $n\equiv 2\MOD 4$ and $n\equiv 4 \MOD 8$ will then be considered in Section \ref{section-11}.  
Further information for the pairs $(3,8)$ and $(3,16)$ is given in Section \ref{section-12}.

\section{A further sufficient existence criterion}  
\label{section-08} 
The aim of the present section is to prove a relaxation of Proposition \ref{prop:suff-crit}, which however is 
easier to apply to almost all pairs $(q,n)$ under consideration. 
It relies on upper bounds for the values $\omega$ and $\Omega^c$. 

\begin{proposition} \label{prop:weak-suff-crit}
Assume that $\fqn$ is a regular extension over $\fq$, where $q\equiv 3\MOD 4$ and $n$ is even. 
Let again $n=p^an'$ with $n'$ indivisible by $p$.  
Suppose that 
\[\frac{16}{n} + \frac{9}{4p^a} \leq \log_2(q),  {\mbox{ if $4\mid n$,}} \] 
and  
\[\frac{16}{n} + \frac{3}{p^a} \leq \log_2(q),  {\mbox{ if $n \equiv 2 \MOD 4$.}}\]   
Then there exists a primitive element in $\fqn$ that is completely normal over $\fq$. 
\end{proposition} 
We have splitted the proof of this result into four subsections. 
After deriving upper bounds $u$ for $\omega$ and $U^c$ for $\Omega^c$ such that 
$u+U^c\leq \frac{n}{2} \log_2(q)$, the existence of a primitive completely normal element in $\fqn/\fq$ is guaranteed by Proposition \ref{prop:suff-crit}. 

\subsection{Upper bounds for $\omega$} \label{subsection-upp-omega}
For an integer $\ell\geq 1$ let $P_\ell$ be the set of all primes $r < \ell$. 
If $\Lambda$ is a subset of $P_\ell$ such that 
$P_\ell \cap \pi(q^n-1) \subseteq \Lambda$, 
then Lemma 2.6 from \cite{Le-Sc-1987} gives  
\[\omega \leq \frac{\log(q^n-1)-\log(L)}{\log(\ell)} + |\Lambda|, \ \ {\mbox{ where }} \ \ L:=\prod_{r \in  \Lambda}r.\] 
For our purposes it turned out that $\ell:=64$ is a convenient choice. Thus, 
$P_\ell:=\{2$, $3$, $5$, $7$, $11$, $13$ ,$17$, $19$, $23$, $29$, $31$, $37$, $41$, $43$, $47$, $53$, $59$, $61\}.$  
Taking $\Lambda=P_\ell$ and the logarithm to the base $2$, we obtain  
$\lfloor \log_2(L) \rfloor = 76$ and therefore 
\[\omega < \frac{n\log_2(q) - 76}{6}+18 = \frac{n}{6} \log_2(q) + \tfrac{16}{3}=:u.\] 
Consequently, if $U^c\geq \Omega^c$, then the condition 
\[\frac{3}{n} \cdot \left(\tfrac{16}{3} + U^c\right) = 
\frac{16}{n} + \frac{3U^c}{n} \leq \log_2(q)\] 
is sufficient for the existence of a primitive completely normal element in $\fqn/\fq$. 

\subsection{Formulas for $\Omega$ and $\Omega^\varepsilon$} \label{subsec-8-2} 
We take up the terminology introduced at the end of Section \ref{section-03}, see (\ref{eqn:3C1})-(\ref{eqn:3C3}). 
So, for a regular pair $(q,n)$ let again $k\mid n'$. 
First of all, 
 \[|F_k|=\frac{\varphi(k/\tau_k)}{\ord_{k/\tau_k}(q^{\tau_k})} 
\ \ 
{\mbox{ and }} |F^\varepsilon_k| = \frac{\varphi(k/(2\tau_k))}{\ord_{k/(2\tau_k)}(q^{\tau_k})} 
\] 
for $k$ from ${\mathcal N}$ or ${\mathcal E}$, respectively. 

\begin{enumerate} 
\item 
By definition of $\tau_k$, the radical of $k/\tau_k$ is the radical of $k$, and therefore 
$\varphi(k/\tau_k)=\varphi(k)/\tau_k$, implying 
\[|F_k|\leq \frac{\varphi(k)}{\tau_k} \ \ {\mbox{ for $k\in {\mathcal N}$.}}\] 
Moreover, if $k\in {\mathcal E}$, then $\tau_k$ is odd while $k\equiv 0 \MOD 8$. Therefore, the radical of $k/(2\tau_k)$ is the radical of $k$. 
This gives $\varphi(k/(2\tau_k))=\varphi(k)/(2\tau_k)$ and implies 
\[|F_k^\varepsilon|\leq  \frac{\varphi(k)}{2\tau_k} \ \ {\mbox{ for $k\in {\mathcal E}$.}}\] 
\item Next, for any $k\mid n'$ we have 
$\ord_{k/\tau_k}(q^{\tau_k}) = \ord_k(q)/\tau_k^2$. 
Moreover, it holds that 
$\ord_{k/(2\tau_k)}(q^{\tau_k}) = \ord_{k/\tau_k}(q^{\tau_k}) = \ord_k(q)/\tau_k^2$ when $k\in {\mathcal E}$. 
This implies 
\begin{equation} \label{eqn:number-components} 
|F_k|= \tau_k \cdot \frac{\varphi(k)}{\ord_k(q)} \ \ {\mbox{ and }} \ \ |F^\varepsilon_k|= \tau_k \cdot \frac{\varphi(k)}{2\cdot \ord_k(q)} 
\end{equation}  
for $k$ from $\mathcal N$ and ${\mathcal E}$, respectively. 
\item 
Additional information can be obtained as follows by using the multiplicativity of the $\tau$- and the $\varphi$-function. 
We now write $k=2^j\ell$ with $\ell$ odd and $j\in \{0,1,\dots,b\}$ 
(recall from the end of Section \ref{section-03} that $n'=2^b\overline{n}$ with $\overline{n}$ being odd, hence $b\geq 1$; 
moreover, $2^e$ is the maximal power of $2$ dividing $q^2-1$, hence $e\geq 3$). 

\begin{enumerate} \item 
Because of the regularity of $(q,k)$ one even has 
$\ord_k(q)= \ord_{2^j}(q) \cdot \ord_{\ell}(q)$  
\item Altogether, this implies 
$|F_k|=|F_{2^j}| \cdot |F_\ell|$ for all $k\in {\mathcal N}$,  that is, $j\leq 2$ or $j \in \{e+1,\dots,b\}$ (for $b>e$ in the latter case). 
Given that $\mathcal E$ is non-empty, we additionally have 
$|F^\varepsilon_k|=|F^\varepsilon_{2^j}| \cdot |F_\ell|$ for all $k\in {\mathcal E}$. 
\item 
Furthermore, 
$1=|F_1|=|F_2|=|F_4|$, and therefore $|F_{2^j\ell}|=|F_\ell|$ when $j\leq 2$. 
If $b>e$ and $j\in \{e+1,\dots,b\}$, then 
\[|F_{2^j}|=\tau_{2^j} \cdot \frac{2^{j-1}}{2^{j-e+1}} = \tau_{2^j} \cdot 2^{e-2}.\] 
Finally, when $b\geq 3$ and $j\in \{3,\dots,\min(e,b)\}$, then 
$|F_{2^j}^\varepsilon|=2^{j-3}$.  
\end{enumerate} 
\item Recall also from (\ref{eqn:3C1}) that $D_j:=\{2^j \ell: \ell\mid \overline{n}\}$ for $j=0,\dots,b$. 
Let \[\Omega_j:=\sum_{k\in D_j} |F_k| {\mbox{ when }} D_j\subseteq {\mathcal N}, \ {\mbox{ while }}  \ 
\Omega^\varepsilon_j:=\sum_{k\in D_j} |F^\varepsilon_k| {\mbox { for }} D_j \subseteq {\mathcal E}.\]  
Then 
\[\Omega_j:=|F_{2^j}| \cdot \sum_{\ell\mid \overline{n}} |F_\ell|=|F_{2^j}|\cdot \Omega_0 {\mbox{ for }} D_j\subseteq {\mathcal N},\]   
and analogously, 
\[\Omega^\varepsilon_j:=|F^\varepsilon_{2^j}| \cdot \Omega_0 {\mbox{ for }} D_j\subseteq {\mathcal E}.\]   
In particular:
\begin{enumerate} 
\item  $\Omega_1=\Omega_0$, and, $\Omega_2=\Omega_0$ provided $b\geq 2$;  
\item if $b\geq 3$ and $j\in \{e+1,\dots,b\}$, then 
$\Omega^\varepsilon_j=2^{j-3} \cdot \Omega_0$. 
\end{enumerate} 

\item Finally, we define  
\[\Omega':=\left\{\begin{array}{ll} 
\Omega_0 + \Omega_1,& {\mbox{ if $b=1$, }} \\ 
\Omega_0 + \Omega_1 + \Omega_2,& {\mbox{ if $b=2$}}\end{array} \right. = \ 
\left\{\begin{array}{ll} 
2\Omega_0,& {\mbox{ if $b=1$, }} \\ 
3\Omega_0,& {\mbox{ if $b=2$}}\end{array} \right. \] 
and 
\[\Omega'':=\sum_{j=e+1}^b \Omega_j = \left(\sum_{j=e+1}^b|F_{2^j}|\right) \cdot \Omega_0 \ \ {\mbox{ for $b>e$.}}\] 
Then, with $\Omega$ and $\Omega^\varepsilon$ as in (\ref{eqn:Omega-eps}), $\Omega = \Omega' + \Omega''$, and 
\[\Omega^\varepsilon=\sum_{j=3}^{\min(e,b)} \Omega_j^\varepsilon =  (2^{\min(e,b)-2}-1) \cdot \Omega_0 \ \ {\mbox{ for $b\geq 3$.}}\] 
\end{enumerate}

\subsection{Upper bounds for $\Omega$ and $\Omega^\varepsilon$} \label{subsection-upp} 
Using the trivial upper bound $|F_\ell|\leq \varphi(\ell)$, we obtain 
$\Omega_0\leq \sum_{\ell\mid \overline{n}} \varphi(\ell)=\overline{n}$  
and Subsection \ref{subsec-8-2} (5) then implies the upper bounds  
\[\Omega' \leq \left\{ \begin{array}{ll} 2\overline{n}, & {\mbox{ if $b=1$,}} \\ 
                                         3\overline{n}, & {\mbox{ if $b\geq 2$}} \end{array} \right. 
                                         \ {\mbox{ as well as }} \ 
                                         \Omega^\varepsilon\leq (2^{\min(e,b)-2}-1) \cdot \overline{n}.\] 
As for $j\in\{e+1,\dots,b\}$ (when $b>e$) we may use the fact that $4$ divides $\ord_{2^j}(q)$, whence $\tau_{2^j}$ is divisible by $2$. Thus, 
$|F_k|\leq \frac{\varphi(k)}{\tau_k} \leq \frac{\varphi(k)}{2} = 2^{j-2}\cdot \varphi(\ell)$ for all $k\in D_j$, and this implies $\Omega_j \leq 2^{j-2} \overline{n}$ for all these $D_j$. Consequently, 
\[\Omega''\leq \sum_{j=e+1}^b 2^{j-2} \overline{n} = 2^{e-1} \cdot (2^{b-e}-1) \cdot \overline{n}.\] 
This altogether gives an upper bound for $\Omega + \Omega^\varepsilon = \Omega' + \Omega^\varepsilon + \Omega''$.  

For the case where $q=3$ better upper bounds will have to be provided in Section \ref{section-10}.

\subsection{Completion of the proof of Proposition \ref{prop:weak-suff-crit}} 
The following upper bounds for $\Omega^c=\Omega+3\Omega^\varepsilon$ (see Proposition \ref{prop:suff-crit}) are immediate from Subsection \ref{subsection-upp}. 
\begin{enumerate} 
\item If $b=1$, then $\Omega^c=\Omega'\leq 2\overline{n}=:U^c$, which gives 
$\frac{3U^c}{n}=\frac{3}{p^a}$. 
\item If $b=2$, then $\Omega^c=\Omega'\leq 3\overline{n}=:U^c$ and therefore 
$\frac{3U^c}{n}=\frac{9}{4p^a}$. 
\item If $3\leq b \leq e$, then 
$\Omega^c = \Omega' + 3 \Omega^\varepsilon \leq U^c$, where 
$U^c:=3\overline{n} + 3\cdot (2^{b-2}-1)\overline{n} = 3\cdot 2^{b-2}\overline{n}$.  
This gives $\frac{3U^c}{n}=\frac{9}{4p^a}$ as well. 
\item If $b>e$, then $\Omega^c = \Omega' + \Omega'' + 3\Omega^\varepsilon$ is less than or equal to 
$$3\overline{n} + 2^{e-1} (2^{b-e}-1)\overline{n} + 3(2^{e-2}-1)\overline{n},$$ which is 
$(2^{b-1}+2^{e-2})\overline{n}$. 
The latter is at most equal to $(2^{b-1}+2^{b-3}) \overline{n} =2^{b-3} \cdot 5 \cdot \overline{n} $. 
Therefore, $\Omega^c\leq U^c :=2^{b-2} \cdot 3 \cdot \overline{n}$. 
This gives once more $\frac{3U^c}{n}=\frac{9}{4p^a}$ and altogether proves the assertion of Proposition \ref{prop:weak-suff-crit}.   \qed 
\end{enumerate}

\section{The case $q>3$ and $n\equiv 0 \MOD 8$}
\label{section-09} 
Throughout, we assume that $n\equiv 0 \MOD 8$ and that $q\equiv 3 \MOD 4$. 
Then, 
\[\frac{16}{n} + \frac{9}{4p^a} \leq 2+\tfrac{9}{4}=\tfrac{17}{4}\] 
and $\lceil 2^{17/4}\rceil =20$. Consequently, the condition in Proposition \ref{prop:weak-suff-crit} is satisfied for all 
$q\geq 20$. It therefore remains to study the cases where 
$q \in \{3,7,11,19\}$.  
We deal with $q=19$ and $q=11$ and $q=7$ here, while $q=3$ is considered in the next section. 
Generally, when $n=8$, we have ${\mathcal N}=\{1,2,4\}$ and ${\mathcal E}=\{8\}$; moreover, $\Omega=\Omega'=3\Omega_0=3$ and $\Omega^\varepsilon=1$, hence 
$\Omega^c=6$ in this case.

\subsection{The case $q=19$:} 
Let first $q=19$. Then $\lfloor \log_2(19)\rfloor =4$. As $a\geq 0$, we have 
$\frac{16}{n} + \frac{9}{4p^a} \leq \frac{16}{n} + \frac{9}{4}$, and this is less than or equal to $4$ whenever 
$n \geq 16/(4-\frac{9}{4})=\frac{64}{7} > 9$.  
For these values of $n$ the condition in Proposition \ref{prop:weak-suff-crit} is satisfied. 
For the remaining case, namely $(q,n)=(19,8)$, we check the condition in Proposition \ref{prop:suff-crit}. 
As remarked above, $\Omega^c=6$. Furthermore, the prime power decomposition 
\[19^8-1=(19^2-1)\cdot (19^2+1) \cdot (19^4+1) = 2^5 \cdot 3^2 \cdot 5 \cdot 17 \cdot 181 \cdot 3833\] 
shows $\omega=6$. 
Now, $\omega+\Omega^c=12$ and $2^{12}=4096 < 130321 = 19^4$. 

\subsection{The case $q=11$:} 
Let next $q=11$. Then $\lfloor \log_2(11)\rfloor = 3$ and as $a\geq 0$, we have that 
$\frac{16}{n} + \frac{9}{4} \leq 3$ implies 
$\frac{16}{n} + \frac{9}{4p^a} \leq \log_2(11)$. 
Consequently, the condition in Proposition \ref{prop:weak-suff-crit} is satisfied whenever  
$n \geq 16/(3-\frac{9}{4})=\frac{64}{3} > 21$.  
Given that $n$ is divisible by $8$, it remains to consider the numbers $n\in \{8,16\}$. 
\begin{enumerate} 
\item For $n=8$ we know that $\Omega^c=6$. Furthermore,  the prime power decomposition of $11^8-1$ is 
\[11^8-1=(11^2-1)\cdot (11^2+1) \cdot (11^4+1) = 2^5 \cdot 3 \cdot 5 \cdot 61 \cdot 7321,\] 
and therefore $\omega=5$. 
Now, the condition in Proposition \ref{prop:suff-crit} is satisfied, because $\omega+\Omega^c=11$ and $2^{11}=2048 < 14641 = 11^4$. 
\item 
For $n=16$ we have $b=4$ and $e=3$ (as $11^2-1=120\equiv 8 \MOD 16$). 
Thus, $\Omega^c=6+|F_{16}|$ (where the summand $6$ comes from the corresponding value for $n=8$). 
As $\tau_{16}=2$, we obtain $|F_{16}|=4$ from the relevant part of Section \ref{section-08}, and therefore 
$\Omega^c=10$. Furthermore, 
\[11^{16}-1=(11^8-1)\cdot (11^8+1) = 2^6 \cdot 3 \cdot 5 \cdot 61 \cdot 7321 \cdot 17 \cdot 6304673\] 
is the prime power decomposition of $11^{16}-1$, 
and therefore $\omega=7$. 
Now, $\omega+\Omega^c=17$ and $2^{17} < (2^{9})^2 = 512^2 < 14641^2 = 11^8$. Again, the condition in Proposition \ref{prop:suff-crit} is satisfied. 
\end{enumerate} 

\subsection{The case $q=7$:} 
Assume finally that $q=7$. Then $\log_2(7) > 2.75$ and as $a\geq 0$, we have that 
$\frac{16}{n} + \frac{9}{4} \leq 2.75$ implies 
$\frac{16}{n} + \frac{9}{4p^a} \leq \log_2(7)$. 
Consequently, the condition in Proposition \ref{prop:weak-suff-crit} is satisfied whenever  
$n \geq 16/(2.75-\frac{9}{4})=32$.  
Given that $n$ is divisible by $8$, it remains to consider the numbers $n\in \{8,16,24,32\}$. 
\begin{enumerate} 
\item For $n=8$ we know that $\Omega^c=6$. Here, we have 
\[7^8-1=(7^2-1)\cdot (7^2+1) \cdot (7^4+1) = 2^6 \cdot 3\cdot 5^2 \cdot 1201,\] 
and therefore $\omega=4$. 
Thus, $\omega+\Omega c=10$ and $2^{10}=1024 < 2401 = 7^4$, and Proposition \ref{prop:suff-crit} gives the existence 
for the pair $(7,8)$. 
\item As $7^2-1=16\cdot 3$, for $n=16$ we have $e=b=4$. 
Consequently $\Omega^c=6+3\cdot |F^\varepsilon_{16}|$ (the first summand comming from the case $n=8$). 
As $|F^\varepsilon_{16}|=2$ (by considerations in Section \ref{section-08}), we obtain $\Omega^c=12$. 
Furthermore, 
\[7^{16}-1=(7^8-1)\cdot (7^8+1) = 2^7 \cdot 3\cdot 5^2 \cdot 1201 \cdot 17 \cdot 169553\] 
is the prime power decomposition of $7^{16}-1$, and therefore $\omega=6$. 
Now, $\omega+\Omega^c=18$ and $2^{18} = 512^2 < 2401^2 = 7^8$, hence the condition in Proposition \ref{prop:suff-crit} 
is satisfied. 
\item For $n=24$ we have ${\mathcal N}={\mathcal N}'=\{1,3,2,6,4,12\}$ and ${\mathcal E}=\{8,24\}$. 
Here, $\Omega'=3\Omega_0=3\cdot (1+2)=9$ and $\Omega^\varepsilon=\Omega_0=3$ (see again Section \ref{section-08}). 
Therefore 
$\Omega^c=18$. Moreover, 
\[\begin{array}{lcl} 
7^{24}-1 &=& (7^3-1)\cdot (7^3+1) \cdot (7^6+1) \cdot (7^{12}+1) \\ 
         &=& 2^6 \cdot 3^2 \cdot 5^2 \cdot 13 \cdot 19 \cdot 43 \cdot 73 \cdot 181 
\cdot 193 \cdot 409 \cdot 1201,\end{array} \] 
is the prime power decomposition of $7^{24}-1$, and therefore $\omega=11$. 
Now, $\omega+\Omega^c=29$ and $2^{29} < 2^{32} = 256^4 < 343^4 = (7^3)^4=7^{12}$. Again, the condition in Proposition \ref{prop:suff-crit} is satisfied. 
\item For $n=32$ we observe that $e=4<5=b$. We here use the upper bound 
\[\Omega^c \leq (2^{b-1}+2^{e-2})\overline{n} = 2^4+2=18\] 
derived at the very end of Section \ref{section-08}. 
Now, 
\[\frac{3}{n} \cdot \left(\tfrac{16}{3}+18\right)=\tfrac{1}{2} + \tfrac{27}{16}=\tfrac{35}{16} < 2.25 < \log_2(7).\] 
This settles the existence for the pair $(q,n)=(7,32)$ as well. 
\end{enumerate}

\section{The case $q=3$ and $n\equiv 0 \MOD 8$}
\label{section-10} 
For the case $q=3$ and $n\equiv 0 \MOD 8$ the bound in Proposition \ref{prop:weak-suff-crit} is only good enough, when $a\geq 1$, 
which means that $3\cdot 8=24$ divides $n$. 
But then 
\[\frac{16}{n} + \frac{9}{4p^a}  \leq \tfrac{2}{3} + \tfrac{3}{4} 
= \tfrac{17}{12} < \tfrac{3}{2} < \log_2(3)\] 
shows that the condition in Proposition \ref{prop:weak-suff-crit} is in fact satisfied. 

It is from now on sufficient to consider the case where $n=n'\equiv 0 \MOD 8$.  

Let us first have a look at $n=8$ and $n=16$. 
\begin{enumerate} 
\item If $n=8$, then $\Omega^c=6$.  
Furthermore, 
$3^8-1=2^5\cdot 5 \cdot 41$ giving $\omega=3$. Now, 
$(2^\omega-1)(2^{\Omega^c}-1)=7\cdot 63 = 441 > 81=3^4$. Thus, the condition in Proposition \ref{prop:suff-crit} is not satisfied. 
\item When $n=16$, then $b=4>3=e$ as $3^2-1=8$. 
Therefore, $\Omega^c=6+|F_{16}|$ (the summand $6$ comming from the factor $k=8$ of $n$). 
As $\tau_{16}=2$, we have $|F_{16}|=4$. Consequently, $\Omega^c=10$ in the present case. 
Furthermore, 
$3^{16}-1=(3^8-1)\cdot (3^8+1) = 2^6\cdot 5 \cdot 41 \cdot 17 \cdot 193$, and this gives $\omega=5$. Now, 
$(2^\omega-1)(2^{\Omega^c}-1)=31\cdot 1023 = 31713 > 6561 = 3^8$, hence the condition in Proposition \ref{prop:suff-crit} is again not satisfied. 
\end{enumerate} 

\noindent 
Assume now that $n\geq 32$ is divisible by $8$ and relatively prime to $3$. 
We first consider the case, where $n$ is a power of $2$, that is, $n=2^b$ with $b\geq 5$. 
Recall that $e=3$ as $3^2-1=8=2^3$. 
\begin{enumerate} 
\item Let $b=5$. Then $|F_{32}|=4$ as $\ord_{32}(3)=8$ and $\tau_{32}=1$. 
Using the calculation for the part where $n=16$, we obtain $\Omega^c=10+4=14$ in the present case. 
As $2 \cdot 21523361$ is the prime power decomposition of $3^{16}+1$, we have $\omega=6$, here. 
Now $2^{\omega+\Omega^c} = 2^{20} = 1024^2 < 6561^2=3^{16}$ and therefore the condition of 
Proposition \ref{prop:suff-crit} is satisfied by the pair $(3,32)$. 
\item Assume $n=2^b$ with $b\geq 6$. As $\ord_{64}(3)=16$, we have that $\tau_{2^j}$ is divisible by $4$ for all 
$j=6,\dots,b$. 
Thus, 
$|F_{2^j}|\leq \varphi(2^j)/\tau_{2^j}  \leq 2^{j-3}$ 
for all these $j$. This gives 
\[\Omega^c\leq 14+ \sum_{j=6}^b 2^{j-3} = 14 + 8\cdot (2^{b-5}-1) = 2^{b-2}+6 =:U^c\] 
(the summand $14$ comming from the previous case for the divisor $k=32$ of $n$). 
As in the proof of Proposition \ref{prop:weak-suff-crit} it suffices now to show that 
\[\log_2(3) \geq \frac{3}{n}\cdot \left( \tfrac{16}{3} + U^c\right).\]  
But $n=2^b$ and therefore the right hand side is 
$$\frac{3}{2^b}\cdot \left( \tfrac{16}{3} + 2^{b-2}+6 \right) = \frac{16+18}{2^b} + \tfrac{3}{4} 
\leq \tfrac{34}{64} + \tfrac{3}{4} = \tfrac{41}{32} < \tfrac{3}{2} < \log_2(3),$$ 
implying the existence of a primitive completely normal element in these extensions. 
\end{enumerate} 

\noindent 
Assume finally that $n=2^b\cdot \overline{n}$ with $b\geq 3$ and $\overline{n}>1$ odd. 
Because of the regularity, we have that $2$ does not divide $\ord_r(3)$ for every prime divisor $r$ of $\overline{n}$. 
As $\ord_5(3)=4$ and $\ord_7(3)=6$, we obtain $r\geq 11$ for every prime divisor $r$ of $\overline{n}$ (in fact 
$\ord_{11}(3)=5$ is odd). In particular $\overline{n}\geq 11$. 
But then $\ord_\ell(3)\geq 3$ for any $\ell\mid \overline{n}$ with $\ell\not=1$. 
Because of the regularity assumption, and as the prime divisors of $\tau_\ell$ divide $\ell$, we obtain 
$\ord_r(q)=\ord_r(q^{\tau_\ell})$ for any $\ell$. As the radical of $\ell$ is equal to the 
radical of $\ell/\tau_\ell$ we even obtain that 
\[\ord_{\ell/\tau_\ell}(q^{\tau_\ell})\geq 3 \ \ {\mbox{ for all $\ell\mid \overline{n}$ with $\ell\not=1$.}}\] 
This implies  
\[|F_\ell|= \frac{\varphi(\ell)}{\tau_\ell \cdot \ord_{\ell/\tau_\ell}(q^{\tau_\ell})} 
\leq \frac{\varphi(\ell)}{\ord_{\ell/\tau_\ell}(q^{\tau_\ell})}\leq \frac{\varphi(\ell)}{3}\] 
for all these $\ell$. 
Therefore, 
\[\Omega_0=\sum_{\ell\mid \overline{n}} |F_\ell| \leq  1 + \sum_{\ell\mid \overline{n}, \ell\not=1} \frac{\varphi(\ell)}{3} = 
1+ \tfrac{1}{3} \cdot (\overline{n}-1)=\tfrac{1}{3}\overline{n}+\tfrac{2}{3}.\] 
This gives 
$\Omega'=3\Omega_0\leq \overline{n} + 2$, and as ${\mathcal E}=D_3=\{8\ell: \ell \mid \overline{n}\}$, we have 
$\Omega^\varepsilon = \Omega_0 \leq \frac{1}{3}\overline{n}+\frac{2}{3}$. 
If $b>e=3$, then for all $j\in \{4,\dots,b\}$ it holds that 
$|F_{2^j}|\leq \varphi(2^j)/\tau_{2^j}  \leq 2^{j-2}$,  
as $\tau_{2^j}$ is divisible by $2$. This implies $\Omega_j\leq 2^{j-2} \Omega_0$ for these $j$ and gives 
\[\begin{array}{lcl} 
\Omega'' &\leq& 2^{e-1} \cdot (2^{b-e}-1) \cdot \Omega_0 \\ 
         &=& 4 \cdot (2^{b-3}-1) \cdot \Omega_0 \\  
         &\leq& 4 \cdot (2^{b-3}-1) \cdot (\frac{1}{3}\overline{n}+\frac{2}{3}) \\ 
         &=& (2^{b-3}-1) \cdot (\frac{4}{3}\overline{n}+\frac{8}{3}). \end{array}\] 
         
\begin{enumerate} 
\item Suppose $b=3$. Then 
$\Omega^c = \Omega' + 3 \Omega^\varepsilon \leq 2\overline{n}+4=:U^c$. 
Therefore 
\[\frac{3}{n} \cdot \left( \tfrac{16}{3} + U^c \right) 
= \frac{3}{n} \cdot \left( \tfrac{16}{3} + 2\overline{n}+4 \right) 
= \frac{20}{n} + \frac{6}{2^b}=\frac{20}{n} + \tfrac{3}{4}.\]  
As $n\geq 8\cdot 11$ is less or equal to 
$\frac{20}{88} + \frac{3}{4} = \frac{43}{44} < 1 < \log_2(3)$,  
the condition of Proposition \ref{prop:suff-crit} can be satisfied in this case. 
\item Suppose next that $b\geq 4$. Then 
$\Omega^c = \Omega' + 3 \Omega^\varepsilon + \Omega'' \leq U^c$, where 
\[U^c:=2\overline{n}+4 + (2^{b-3}-1) \cdot \left(\tfrac{4}{3}\overline{n}+\tfrac{8}{3}\right)  = 
\tfrac{2}{3} \overline{n} + 2^{b-3} \overline{n} +\tfrac{4}{3} +\frac{2^b}{3} .\] 
Here, 
$$\frac{3}{n} \cdot \left( \tfrac{16}{3} + U^c\right) = \frac{20}{n} + \frac{1}{2^{b-1}} + \tfrac{3}{8} + \frac{1}{\overline{n}}.$$ 
Moreover, $n\geq 16\cdot 11$ as $b\geq 4$ and $\overline{n}\geq 11$. 
Therefore,  
$$\frac{3}{n} \cdot \left( \tfrac{16}{3} + U^c\right) \leq \tfrac{5}{44} + \tfrac{1}{8} + \tfrac{3}{8} + \tfrac{1}{11} = \tfrac{41}{44} 
< 1 < \log_2(3).$$ 
This shows that the condition in Proposition \ref{prop:suff-crit} can also be satisfied for these parameters. 
\end{enumerate} 

\noindent 
The proof of Proposition \ref{prop:evaluation} is now complete.

\section{The case $q\equiv 3 \MOD 4$, and $n\equiv 2 \MOD 4$ or $n\equiv 4 \MOD 8$}
\label{section-11} 
In this section we finally settle the existence of primitive completely normal elements in 
regular extensions 
$\fqn/\fq$ where $q\equiv 3 \MOD 4$ and $n$ is even, but not divisible by $8$. 
Here we have that the set $\mathcal E$ of exceptional indices is empty, while 
${\mathcal N}={\mathcal N}'$ and therefore $\Omega^c=\Omega'$. 
We write again $n=p^a \cdot 2^b \cdot \overline{n}$ with $\overline{n}$ odd. 
Now $b=1$ or $b=2$. 
If $b=1$, then $\Omega^c=2\Omega_0$, while $\Omega^c=3\Omega_0$ when 
$b=2$ (see Subsection \ref{subsec-8-2} (5)). 

Similar to \cite[Section 6]{Ha-2001} we do not seek a classification of all 
pairs $(q,n)$ which satisfy the sufficient condition in Proposition \ref{prop:suff-crit}. 
Instead, we exclude those pairs $(q,n)$ for which the extension $\fqn/\fq$ is completely basic in advance and 
are therefore able to work with better estimates for $\Omega^c$: 
A Galois field extension $\fqn/\fq$, and also the pair $(q,n)$ are called \textit{completely basic}, if every normal element of $\fqn/\fq$ 
already is completely normal in $\fqn/\fq$.  
According to \cite[Theorem 3.1]{Ha-2001}, 
for an extension $\fqn/\fq$, the following are equivalent:  
\begin{enumerate} 
\item $\fqn$ is completely basic over $\fq$. 
\item For every prime divisor $r$ of $n$, every normal element of $\fqn/\fq$ is normal in $\fqn/{\mathbb F}_{q^r}$. 
\item For every prime divisor $r$ of $n$, the number ${\rm ord} _{(n/r)'}(q)$ is not divisible by $r$. 
\end{enumerate} 
Furthermore, by \cite[Proposition 3.2]{Ha-2001}, every completely basic extension is regular. 
On the other hand, assuming that $(q,n)$ is regular, then $(q,n)$ is completely basic if and only if 
$(q,n)$ is not exceptional and $\alpha(r)\leq 1$ for every prime divisor $r$ of $n'$ (this is \cite[Proposition 3.3]{Ha-2001}). 
Here, $\alpha(r)$ is the parameter occuring in the suborder of $q$ modulo $n'$, see (\ref{eqn:3B2}) in Section \ref{section-04}. 
In other words, if $(q,n)$ is regular but not completely basic, then there is a prime divisor $r$ of $n'$ such that 
$r^2$ divides $\ord_{n'}(q)$. Because of the regularity-condition, even $r^3$ divides $n'$, then. 
Since $n \not\equiv 0 \MOD 8$ we may from now on assume that $n'$ is divisible by the cube of an odd prime $r$.  
Moreover, one has that $\tau_k$ is divisible by $r$ for any $k\mid n'$ such that $r^3\mid k$. 
It is proved in \cite[Section 6, see formula (6.6)]{Ha-2001} 
that then (with the present notation) 
\[\Omega^c \leq \frac{2r-1}{r^2} \cdot n'=:U^c\] 
for all these situations. 
Now, taking the same upper bound $u$ for the number of distinct prime divisors of $q^n-1$ as in Subsection \ref{subsection-upp-omega}, namely  
\[u:=\frac{n}{6} \log_2(q) + \tfrac{16}{3},\] 
we have that the condition 
\[\frac{3}{n}\cdot \left( \tfrac{16}{3} +  \frac{2r-1}{r^2} \cdot n' \right) \leq \log_2(q)\] 
is sufficent for the existence of a primitive completely normal element in $\fqn/\fq$. 
The left hand side of the latter inequality is 
\[\frac{16}{n} +  \frac{6r-3}{r^2p^a}.\] 
Now, $n\geq 2 r^3$ and $p^a\geq 1$. Furthermore, the function $r\mapsto \frac{6r-3}{r^2}$ is monotonely decreasing for $r\geq 1$. 
Therefore  
\[\frac{16}{n} +  \frac{6r-3}{r^2p^a} \leq \frac{16}{2r^3} +  \frac{6r-3}{r^2} \leq \tfrac{8}{27} +\tfrac{15}{9} = \tfrac{53}{27} < 2.\]
This settles the existence for all $q\geq 4$ and it therefore remains to consider the case where $q=3$. 
But when $q=3$, then $r\geq 11$ for every odd prime divisor of $n'$, as $(3,n)$ is regular ($\ord_r(3)$ must not be divisible by $2$). 
Therefore, 
\[\frac{16}{n} +  \frac{6r-3}{r^2p^a} \leq \frac{16}{2r^3} +  \frac{6r-3}{r^2} \leq \tfrac{8}{11^3} +\tfrac{63}{11^2} = \tfrac{701}{1331} < 1,\] 
and the existence also follows for the case where $q=3$. 

\section{The pairs $(3,8)$ and $(3,16)$}
\label{section-12} 
In the present section we just like to give some information on 
the $8$- and the $16$-dimensional extension of the ternary field $\mathbb{F}_3$. 

Consider first the pair $(3,8)$. The largest power of $2$ dividing $3^8-1$ is $2^5$. 
Over $\mathbb{F}_3$ the cyclotomic polynomial $\Phi_{32}(x)$ splits as 
$(x^8+x^4-1)(x^8-x^4-1)$. Let $\zeta$ be a primitive $32$nd root of unity. 
Based on the theory in \cite[Chapter VI, in particular Section 23]{Ha-1997}, 
$\zeta + \zeta^3$ is a complete generator of the cyclotomic module ${\mathcal C}_8$ over $\mathbb{F}_3$, 
while $\zeta^2$ is a complete generator of ${\mathcal C}_4$, and $\zeta^4\in \mathbb{F}_9$ is normal over $\mathbb{F}_3$. 
By Proposition \ref{prop:cn-by-can-dec}, 
$v:=\zeta^4 + \zeta^2 + (\zeta + \zeta^3)$ therefore is a completely normal element of 
$\mathbb{F}_{3^8}$ over $\mathbb{F}_3$. 
Now, if $\zeta$ in particular is a root of $x^8+x^4-1$, then 
$v$ is also a primitive element of $\mathbb{F}_{3^8}$. The latter has been checked with a computer. 

The field $\mathbb{F}_{3^{16}}$ is obtained from $\mathbb{F}_3$ by adjoining a primitive $64$th root of unity, say $\eta$. 
Then $u:=\eta + \eta^3 + \eta^5 + \eta^7$ is a complete generator of the cyclotomic module ${\mathcal C}_{16}$ over $\mathbb{F}_3$ 
(\cite{Ha-1997}), and therefore, $v+u$ is a completely normal element of $\mathbb{F}_{3^{16}}$ over $\mathbb{F}_3$, when 
$v$ as above is composed by certain powers of a primitive $32$nd root of unity $\zeta$. 
If $\eta$ is a root of $x^{16}+x^8-1$, an irreducible divisor of $\Phi_{64}(x)$ from $\mathbb{F}_3[x]$, and if $\zeta=\eta^2$, then 
$v+u$ is even a primitive element of $\mathbb{F}_{3^{16}}$. Again, the primitivity condition has been 
checked by a computer.

\section{The last step and the number of completely normal elements in regular extensions}
\label{section-13} 
The last step in order to complete the proof of Theorem \ref{thm:main} is to justify what 
has been said in Remark \ref{rem:mixed-orders}. So, let us take up the terminology introduced there, as well as at the beginning of Section \ref{section-06}. 

For an element $w$ of $W_{k,f}$ we consider the following three homomorphisms: 
\begin{enumerate}
\item $\Psi'_K:K[y] \rightarrow E$, $a(y) \mapsto a(S^2)(w)$, 
\item $\Psi_L:L[y] \rightarrow E$, $b(y) \mapsto b(S^2)(w)$, 
\item $\Psi_K:K[x] \rightarrow E$, $c(x) \mapsto c(S)(w)$. 
\end{enumerate} 
Suppose that $w$ has $Q^2$-order equal to $h_1^\alpha h_2^\beta$, where, without loss of generality, $\alpha \geq \beta$. 
Then $h_1^\alpha h_2^\beta$ generates the kernel of $\Psi_L$, and therefore, the kernel of $\Psi'_K$ is generated by $f^\alpha$, since 
$f=h_1h_2$ over $L$. 
This shows that $f(x^2)^\alpha$ is a member of the kernel of $\Psi_K$; the latter 
is generated by the $Q$-order of $w$, say $g_1^\gamma 
g_2^\lambda$. 

The image of $\Psi_K'$ is contained in the image of $\Psi_L$ as well as in the image of 
$\Psi_K$; let these three $K$-vector spaces be denoted by $\im(\Psi_K')$, $\im(\Psi_L)$ and $\im(\Psi_K)$, respectively. 
The $K$-dimension of $\im(\Psi_K')$ is $\alpha \cdot \deg(f)$, while the $K$-dimension of $\im(\Psi_L)$ is equal to 
$2 \cdot(\alpha + \beta) \cdot \deg(f)/2=(\alpha + \beta) \cdot \deg(f)$. 
The $K$-dimension of $\im(\Psi_K)$ is $\gamma \cdot \deg(g_1) + \lambda \cdot \deg(g_2)=(\gamma + \lambda ) \cdot \deg(f)$, 
and this is less than or equal to $2 \alpha \cdot \deg(f)$, since $g_1^\gamma 
g_2^\lambda$ divides $f(x^2)^\alpha$.  

\begin{itemize} 
\item Now, suppose that $\beta=0$. Then $M_{2\tau_k,h_1^\alpha}=\im(\Psi_L) = \im(\Psi_K') \subseteq \im(\Psi_K)$. 
Since $\im(\Psi_K)$ is $S$-invariant, it contains $S(w)$, which has $Q^2$-order equal to $h_2^\alpha$, as well, 
and therefore also $M_{2\tau_k,h_2^\alpha} \subseteq \im(\Psi_K)$. 
This altogether implies that $M_{2\tau_k,f^\alpha}\subseteq \im(\Psi_K)$ (compare with (\ref{eqn:Wkf-as-L})). 
But $M_{2\tau_k,f^\alpha}$ is equal to $M_{\tau_k,f(x^2)^\alpha}$, and this gives that 
the $K$-dimension of 
$\im(\Psi_K)$ is at least $2\alpha \cdot \deg(f)$. Consequently, $\im(\Psi_K)$ has $K$-dimension 
exactly equal to $\alpha \cdot \deg(f(x^2))$, and this means that the $Q$-order of $w$ is equal to 
$f(x^2)^\alpha$. 

\item Assume next that $0\leq \beta < \alpha$, and write $w$ as $v_1+v_2$ where 
$v_1$ has $Q^2$-order $h_1^\alpha$ and $v_2$ has $Q^2$-order $h_2^\beta$. 
By the discussion of the case "$\beta=0$", we obtain that 
$v_1$ then has $Q$-order $f(x^2)^\alpha$, while $v_2$ must have $Q$-order $f(x^2)^\beta$. 
But, since $\beta<\alpha$, the $Q$-order of $w$ is then equal to 
$f(x^2)^\alpha$, and this altogether proves one part of the claim in Remark \ref{rem:mixed-orders}.  
\end{itemize} 

\noindent For the second part of the claim, we assume that $w\in W_{k,f}$ has $Q$-order equal to 
$g_1^\gamma g_2^\lambda$ with $\gamma \not= \lambda$, without loss of generality, let $\gamma > \lambda$. 
From the assertion of the first part of the claim, the $Q^2$-order of $w$ then has to have the form $h_1^\alpha h_2^\alpha = f^\alpha$ for some $\alpha$. 
Therefore, 
$g_1^\gamma g_2^\lambda$ divides $f(x^2)^\alpha=g_1^\alpha g_2^\alpha$, and this shows $\gamma \leq \alpha$. 
On the other hand, $g_1^\gamma g_2^\lambda$ divides $f(x^2)^\gamma$, and therefore 
$f(S^2)^\gamma(w)=0$ gives $\alpha=\gamma$ as $\Ord_{Q^2}(w)=f^\alpha$. 
The latter holds in particular when $\lambda=0$, and this 
settles the second part of the claim in Remark \ref{rem:mixed-orders}. \qed

\vspace{1ex} \noindent 
As mentioned in \cite{Ha-1997} (see Section 21, in particular p. 125), we are now able to derive the following formula for the 
total number of completely normal elements in any regular extension. 

\begin{theorem} 
Assume that $\fqn/\fq$ is a regular extension. Let $n=p^an'$, where $p$ is the characteristic of these fields and $n'$ is the $p$-free part of $n$. 
Let $\mathcal N$ and $\mathcal E$ be the 
index sets for the non-exceptional, respectively exceptional cyclotomic modules of that extension. Then, with $\tau_k$ as defined in (\ref{eqn:3B3}) for any $k\mid n'$,  
the number of completely normal elements of $\fqn/\fq$ is equal to the product of 
\[ \prod_{k\in {\mathcal N}} \left(q^{\ord_k(q)/\tau_k}-1\right)^{\tau_k \varphi(k)/\ord_k(q)} \cdot q^{(p^a-1)\cdot \varphi(k)} \] 
with 
\[
\prod_{k\in {\mathcal E}} \left(q^{2\ord_k(q)/\tau_k}-4q^{\ord_k(q)/\tau_k} +3\right)^{\tau_k \varphi(k)/(2\ord_k(q))} \cdot q^{(p^a-1)\cdot \varphi(k)},\] 
where the second factor is defined to be equal to $1$ provided that $\mathcal E$ is empty. \qed 
\end{theorem} 

\noindent 
For example, the number of completely normal elements of $\mathbb{F}_{3^8}$ over $\mathbb{F}_3$ is equal 
to $(3-1) \cdot (3-1) \cdot (3^2-1) \cdot (3^4-4\cdot 3^2 +3)=1536$, while the number of completely normal elements of 
$\mathbb{F}_{3^{16}}$ over $\mathbb{F}_3$ is equal to 
$1536 \cdot (3^2-1)^4 = 6291456$.

\section{Concluding remarks}
\label{section-14} 
A draft of the present work has already been written in 2014. 
In the meantime further progress concerning the conjecture of Morgan and Mullen has been achieved. 
For an overview, we refer to Section 13.11 of the forthcoming monograph 
{\em Topics in Galois Fields} by Dirk Hachenberger and Dieter Jungnickel, to be published in 2020 
by Springer. For an extensive improvement of the computational results of Morgan and Mullen, as well as 
for a further overview on the state of the art of the conjecture of Morgan and Mullen, we refer 
to the work {\em Computational results on the existence of primitive complete normal basis generators}, 
by Dirk Hachenberger and Stefan Hackenberg, which will soon be available, here in \texttt{arXiv}.

\vspace{4ex} \noindent 
{\bf Acknowledgements.} 
I thank Stefan Hackenberg, a former master student of mine, for independently checking my computational results concerning 
the pairs $(3,8)$ and $(3,16)$.

\vspace{16ex} 

\end{document}